\documentclass[reqno]{amsart}
\usepackage{amsmath, amsthm, amssymb, amstext}

\usepackage[left=2.9cm,right=2.9cm,top=3cm,bottom=3cm]{geometry}
\usepackage{hyperref,xcolor}
\hypersetup{pdfborder={0 0 0},colorlinks}
\usepackage{enumitem}
\allowdisplaybreaks
\usepackage{todonotes}

\newtheorem{theorem}{Theorem}
\newtheorem{remark}[theorem]{Remark}
\newtheorem{lemma}[theorem]{Lemma}
\newtheorem{proposition}[theorem]{Proposition}
\newtheorem{corollary}[theorem]{Corollary}

\newcommand*\diff{\mathrm{d}}
\newcommand{\R}{\mathbb{R}}
\newcommand{\N}{\mathbb{N}}
\newcommand{\Z}{\mathbb{Z}}

\newcommand\V{W^{1,\mathcal{H}}(\R^N)}
\DeclareMathOperator{\supp}{supp}
\DeclareMathOperator{\loc}{loc}

\numberwithin{theorem}{section}
\numberwithin{equation}{section}

\title[Existence of ground state solutions for a Choquard double phase problem]
{Existence of ground state solutions for a Choquard double phase problem}

\author[R. Arora]{Rakesh Arora}
\address[R. Arora]{Department of Mathematical Sciences, Indian Institute of Technology  Varanasi (IIT-BHU), Uttar Pradesh-221005, India}
\email{rakesh.mat@iitbhu.ac.in, arora.npde@gmail.com}

\author[A. Fiscella]{Alessio Fiscella}
\address[A. Fiscella]{Dipartimento di Matematica e Applicazioni, Universit\`a degli Studi di Milano-Bicocca, Via Cozzi 55, Milano, CAP 20125, Italy}
\email{alessio.fiscella@unimib.it}

\author[T. Mukherjee]{Tuhina Mukherjee}
\address[T. Mukherjee]{Department of Mathematics, Indian Institute of Technology Jodhpur, Rajasthan-506004, India-342037}
\email{tuhina@iitj.ac.in}

\author[P. Winkert]{Patrick Winkert}
\address[P. Winkert]{Technische Universit\"{a}t Berlin, Institut f\"{u}r Mathematik, Stra\ss e des 17. Juni 136, 10623 Berlin, Germany}
\email{winkert@math.tu-berlin.de}

\begin{document}

\begin{abstract}
	In this paper we study quasilinear elliptic equations driven by the double phase operator involving a Choquard term of the form
	\begin{align*}
		-\mathcal{L}_{p,q}^{a}(u) + |u|^{p-2}u+ a(x)  |u|^{q-2}u = \left( \int_{\R^N} \frac{F(y, u)}{|x-y|^\mu}\,\diff y\right)f(x,u)
		\quad\text{in } \R^N,
	\end{align*}
	where $\mathcal{L}_{p,q}^{a}$ is the double phase operator given by
	\begin{align*}
		\mathcal{L}_{p,q}^{a}(u):= \operatorname{div}\big(|\nabla u|^{p-2}\nabla u + a(x) |\nabla u|^{q-2}\nabla u \big), \quad u\in\V,
	\end{align*}
	$0<\mu<N$, $1<p<N$, $p<q<p+ \frac{\alpha p}{N}$, $0 \leq  a(\cdot)\in C^{0,\alpha}(\R^N)$ with $\alpha \in (0,1]$ and $f\colon\R^N\times\R\to\R$ is a continuous function that satisfies a subcritical growth. Based on the Hardy-Littlewood-Sobolev inequality, the Nehari manifold and variational tools, we prove the existence of ground state solutions of such problems under different assumptions on the data.
\end{abstract}


\maketitle

\section{Introduction}

In this paper, we are concerned with the existence of ground state solutions to the following double phase Choquard equation
\begin{align}\label{problem}
	-\mathcal{L}_{p,q}^{a}(u) + |u|^{p-2}u+ a(x)  |u|^{q-2}u = \left( \int_{\R^N} \frac{F(y,u)}{|x-y|^\mu}\,\diff y\right)f(x, u)
	\quad\text{in } \R^N,
\end{align}
where $\mathcal{L}_{p,q}^{a}$ is the so-called double phase operator given by
\begin{align}\label{operator_double_phase}
	\mathcal{L}_{p,q}^{a}(u):= \operatorname{div} \big(|\nabla u|^{p-2}\nabla u + a(x) |\nabla u|^{q-2}\nabla u \big), \quad u\in\V,
\end{align}
with $\V$ being an appropriate Musielak-Orlicz Sobolev space and $F$ denotes the primitive of the function $f \in C(\R^N\times\R)$ satisfying 
\begin{align}\label{new:nonl:growth}
		|f(x,t)|\leq C(|t|^{r_1-1}+ |t|^{r_2-1}) \ \text{for all} \ x \in \mathbb{R}^N \ \text{and} \ t \in \R
\end{align}
where 
\begin{align}\label{growth:exponent:cond}
	\frac{p(2N-\mu)}{2N} <r_1\leq r_2< \frac{p^*}{2}\left(2-\frac{\mu}{N}\right), \quad 0< \mu < N,
\end{align} 
and $p^*$ being the critical Sobolev exponent to $p$. The double phase operator given in \eqref{operator_double_phase} is related to the energy functional 
\begin{align}\label{integral_minimizer}
	\Phi(u)=\int_\Omega \big(|\nabla  u|^p+a(x)|\nabla  u|^q\big)\,\diff x,
\end{align}
which appeared for the first time in a work of Zhikov \cite{Zhikov-1986} in order to describe models for strongly anisotropic materials in the context of homogenization and elasticity, see also \cite{Zhikov-1995,Zhikov-2011} by the same author. A first mathematical treatment of double phase integrals like \eqref{integral_minimizer} has been done by Mingione et al.\,concerning regularity results for local minimizers of \eqref{integral_minimizer}. We refer to the works of Baroni-Colombo-Mingione \cite{Baroni-Colombo-Mingione-2015,Baroni-Colombo-Mingione-2016,Baroni-Colombo-Mingione-2018} and Colombo-Mingione \cite{Colombo-Mingione-2015a, Colombo-Mingione-2015b}, see also the recent paper of De Filippis-Mingione \cite{DeFilippis-Mingione-ARMA-2021} about nonautonomous integrals. Note that \eqref{integral_minimizer} also belongs to the class of the integral functionals with nonstandard growth condition, according to Marcellini’s terminology \cite{Marcellini-1991,Marcellini-1989b}.

The stationary Choquard equation 
\begin{align}\label{problem2}
	-\Delta u+V(x)u=\left( \int_{\R^N} \frac{|u(y)|^p}{|x-y|^\mu}\diff y\right)|u|^{p-2}u\quad\text{in }\R^N,
\end{align}
with $N\geq 3$ and $\mu \in (0,N)$ has many physical applications in quantum theory and arises in the theory of the Bose-Einstein condensation. See Lieb \cite{Lieb-1976-77} for an approximation to the Hartree-Fock theory of one-component plasma and Pekar \cite{Pekar-1954} for the study of the quantum theory of a polaron at rest. We also mention the work of Moroz-Penrose-Tod \cite{Moroz-Penrose-Tod-1997} in which \eqref{problem2} serves as a model of self-gravitating matter, known in that context as the Schr\"odinger-Newton equation. A useful guide on Choquard type equations  has been published by Moroz-Van Schaftingen \cite{Moroz-VanSchaftingen-2017}. 

Motivated by these numerous applications, lots of existence results for different type of equations involving Choquard terms have been published in the last decades. We refer to the famous works of Cingolani-Clapp-Secchi \cite{Cingolani-Clapp-Secchi-2012}, Lions \cite{Lions-1980}, Ma-Zhao \cite{Ma-Zhao-2010} and Moroz-Van Schaftingen \cite{Moroz-VanSchaftingen-2015,Moroz-VanSchaftingen-2013, Moroz-VanSchaftingen-2015b}. The treatment in our paper uses ideas of the papers of Alves-Yang \cite{Alves-Yang-2014} and Alves-Tavares \cite{Alves-Tavares-2019}. Indeed, in \cite{Alves-Yang-2014} the authors study a generalized Choquard equation given by
\begin{align}\label{problem3}
	-\varepsilon^p \Delta_p u+V(x)|u|^{p-2}u
	=\varepsilon^{\mu-N} \left(\int_{\R^N}\frac{Q(y)F(u(y))}{|x-y|^{\mu}}\,\diff y\right)Q(x)f(u) \quad\text{in }\R^N,
\end{align}
and establish a new concentration behavior of solutions of \eqref{problem3} by using variational methods. In \cite{Alves-Tavares-2019} a new Hardy-Littlewood-Sobolev inequality for variable exponents has been proved and applied to problems of the form 
\begin{align}\label{problem4}
	-\Delta_{p(x)} u+V(x)|u|^{p(x)-2}u
	= \left(\int_{\R^N}\frac{F(y,u(y))}{|x-y|^{\lambda(x,y)}}\,\diff y\right)f(x,u(x)) \quad\text{in }\R^N,
\end{align}
in order to get existence of solutions of \eqref{problem4}. Very recently, Sun-Chang \cite{Sun-Chang-2022} considered least energy nodal solutions for double phase problems with convolution-type nonlinearities of the form
\begin{equation*}
	\begin{aligned}
		-\operatorname{div}\big(|\nabla u|^{p-2}\nabla u + a(x) |\nabla u|^{q-2}\nabla u \big)&=\left(I_\mu\ast|u|^{r} \right)|u|^{r-2}u
		&&\text{in } \Omega,\\
		u&=0 &&\text{on } \partial\Omega,
	\end{aligned}
\end{equation*}
which becomes the classical nonlinear Choquard equation if $p=2$ and $a(x)\equiv 0$. Further existence results on Choquard type problems can be found in the papers of Alves-Gao-Squassina \cite{Alves-Gao-Squassina-Yang-2017}, Alves-Germano \cite{Alves-Germano-2018}, Arora-Giacomoni-Mukherjee-Sreenadh \cite{Arora-Giacomoni-Mukherjee-Sreenadh-2019,Arora-Giacomoni-Mukherjee-Sreenadh-2020}, Biswas-Tiwari \cite{Biswas-Tiwari-2021}, Ghimenti-Van Schaftingen \cite{Ghimenti-VanSchaftingen-2016}, Mingqi-R\u{a}dulescu-Zhang \cite{Mingqi-Radulescu-Zhang-2019}, Mukherjee-Sreenadh \cite{Mukherjee-Sreenadh-2017} and Zuo-Fiscella-Bahrouni \cite{Zuo-Fiscella-Bahrouni-2022}, see also the works of Chen-Fiscella-Pucci-Tang \cite{Chen-Fiscella-Pucci-Tang-2020} and Liu \cite{Liu-2010} for ground state solution type results. For double phase problems without Choquard term on the whole of $\R^N$ we refer to the recent works of Ge-Pucci \cite{Ge-Pucci-2022}, Le \cite{Le-2022}, Liu-Dai \cite{Liu-Dai-2020}, Liu-Winkert \cite{Liu-Winkert-2022} and Stegli\'{n}ski \cite{Steglinski-2022}, see also Hou-Ge-Zhang-Wang \cite{Hou-Ge-Zhang-Wang-2020} and Liu-Dai \cite{Liu-Dai-2018} for ground state solutions for double phase problems on bounded domains. 

In the present paper, we combine a double phase problem with a right-hand side of Choquard type nonlinearity and we look for ground state solutions. We consider the following assumptions:

\vspace{0.1cm}
\begin{enumerate}
	\item[\textnormal{(h$_1$)}]
		$N\geq 2$, $1<p<N$, $p<q<p+ \frac{\alpha p}{N}$ and $0 \leq  a(\cdot)\in C^{0,\alpha}(\R^N)$ with $\alpha \in (0,1]$ and $p^*$ being the critical Sobolev exponent to $p.$ 
\end{enumerate}

\begin{enumerate}
	\item[\textnormal{(h$_2$)}] 
The function $f$ satisfies the classical Ambrosetti-Rabinowitz condition  (AR-condition for short), that is,
		\begin{align}\label{AR-cond}
			0<\theta F(x, t) \leq 2tf(x, t) 
			\quad \text{for all }t>0, \ x \in \R^N \text{ and for some } \theta>q
		\end{align}
where  $F(x,t)= \int_0^t f(x,\tau)\,\diff \tau.$		

\end{enumerate}

\begin{enumerate}
	\item[\textnormal{(h$_3$)}]
		The modulating coefficient $a(\cdot)$ and the function $f(\cdot, t)$ is $\mathbb{Z}^N$ periodic for all $t \in \mathbb{R}$, that is,
		\begin{align*}
			a(x+y) = a(x)\quad \text{for all }  x \in \R^N  \text{ and for all }  y \in \mathbb{Z}^N,
		\end{align*}
		and 
\begin{align*}
	f(x +y,t) = f(x,t)
	\quad \text{for all } x \in \R^N  \text{ and for all }  y \in \Z^N.
\end{align*}
	\item[\textnormal{(h$_4$)}] 
		The mapping $\R \ni t \mapsto \dfrac{f(x,t)}{|t|^{\frac{q}{2}-2} t}$  is increasing if $t>0$ and decreasing if $t<0$ for all $x \in \R$.
\end{enumerate}
\vspace{0.1cm}

The first existence result of this paper reads as follows.

\begin{theorem}\label{main_result}
	Let hypotheses \textnormal{(h$_1$)}-\textnormal{(h$_3$)}  be satisfied. Then problem \eqref{problem} admits a nontrivial solution $v \in W^{1, \mathcal{H}}(\Omega)$. If in addition \textnormal{(h$_4$)} holds, then $v$ is a ground state solution of problem \eqref{problem}.
\end{theorem}

In order to consider a large class of nonlinearities in the Choquard term, we weaken the hypothesis in the above result and replace our assumptions \textnormal{(h$_2$)} and \textnormal{(h$_4$)} by the following  assumption:

\begin{enumerate}
  \item[\textnormal{(h$_2'$)}]
		\begin{enumerate}
			\item[\textnormal{(i)}]
				There exists $\Theta \geq 1$ such that
				\begin{align}\label{Non-AR-cond}
					\Theta \mathcal{F}(x,t) \geq \mathcal{F}(x,st)
					\quad \text{ for all } t \in \R 
					\text{ and for all } s \in [0,1],
				\end{align}
				where $\mathcal{F}(x,t) = 2f(x,t) t - q F(x,t)$.
			\item[\textnormal{(ii)}] 
				$\lim_{|t| \to \infty} \dfrac{F(x,t)}{|t|^\frac{q}{2}} = \infty$ uniformly for all $x \in \R^N$.
			\end{enumerate}
\end{enumerate}
\vspace{0.1cm}

The assumption \textnormal{(h$_2'$)}(i) is originally due to Jeanjean \cite{Jeanjean-1999} in the case $p=q=2.$ It is important to note that the assumptions \textnormal{(h$_1$)}, \textnormal{(h$_2'$)} and \textnormal{(h$_3$)} allow us to consider a bigger class of nonlinearities, in particular outside the class of functions satisfying Ambrosetti-Rabinowitz condition \eqref{AR-cond}. An example of such function is $f(x,t) = g(x) |t|^{\frac{q}{2}-1} t \ln(1+ |t|)$ where $g$ is a $\mathbb{Z}^N$ periodic bounded function. The Ambrosetti–Rabinowitz condition ensures that the corresponding Euler-Lagrange functional has the mountain pass geometry structure and the associated Palais-Smale sequence of the functional is bounded. Therefore, relaxing
AR -condition \eqref{AR-cond} not only includes a larger class of nonlinearities but also requires a careful geometrical analysis of corresponding Euler-Lagrange functional and compactness results. We make following remarks in light of assumptions \textnormal{(h$_1$)}, \textnormal{(h$_2'$)} and \textnormal{(h$_3$)}.

\begin{remark}
	Note that \eqref{growth:exponent:cond} can be equivalently written as
	\begin{align}\label{ineq:r1r2}
		p< \frac{2Nr_1}{2N-\mu} \leq \frac{2Nr_2}{2N-\mu}< p^*.
	\end{align}
\end{remark}

\begin{remark}\label{rema:1}
	Since $f(x,0)=0= F(x,0)$, thanks to \eqref{new:nonl:growth}, from \eqref{Non-AR-cond}, we get $\mathcal{F}(x,t) \geq 0$, that is,
	\begin{align}\label{non-neg:cond-1}
		2 s f(x,t) - q F(x,t) \geq 0
		\quad \text{for all } t \in \mathbb{R}.
	\end{align}
\end{remark}

\begin{remark}\label{rema:2}
	For $t>0$, using \eqref{non-neg:cond-1}, we have
	\begin{align*}
		\frac{\partial }{\partial t} \frac{F(x,t)}{t^{\frac{q}{2}}} = \frac{t^{\frac{q}{2}} f(x,t) - \frac{q}{2} t^{\frac{q}{2} -1} F(x,t)}{t^q} \geq 0.
	\end{align*}
	Moreover, from \eqref{new:nonl:growth}, we easily see that $\lim_{t \to 0^+} F(x,t)t^{\frac{-q}{2}} =0$. Combining these facts, we get $F(x,t) \geq 0$ for all $(x,t) \in \R^N \times \R$ with $t \geq 0.$ Repeating the same steps as above for $t\leq 0$, we obtain $F(x,t) \geq 0$ for all $t\in \R$. Thus, it holds
	\begin{align*}
		f(x,t) &\geq 0 
		\quad \text{for all } (x,t) \in \R^N \times \R \text{ with }t\geq 0,\\ 
		f(x,t) &\leq 0 
		\quad \text{for all }  (x,t) \in \R^N \times \R
		\text{ with }t\leq 0.
	\end{align*}
	Therefore, for all $x \in \R^N$, we obtain that $F(x,\cdot)$ is nondecreasing in $(0, \infty)$ and nonincreasing in $(-\infty,0).$
\end{remark}

Now we state our second result concerning the existence of ground state solution:

\begin{theorem}\label{main_result-new}
	Let hypotheses \textnormal{(h$_1$)}, \textnormal{(h$_2'$)} and \textnormal{(h$_3$)}  be satisfied. Then problem \eqref{problem} admits a nontrivial ground state solution.
\end{theorem}

\begin{remark}
	It is easy to see that \textnormal{(h$_4$)} implies \textnormal{(h$_2'$)(i)}. Indeed, for $0< t_1 \leq t_2$, we have
	\begin{align*}
		&\mathcal{F}(x,t_2) - \mathcal{F}(x,t_1)\\
		& = q \left[ \frac{2}{q} \left(f(x,t_2) t_2 - f(x,t_1) t_1\right) - \left(F(x,t_2) - F(x,t_1)\right) \right]\\
		& = q \left[\int_{0}^{t_2} \frac{f(x,t_2) \tau^{\frac{q}{2}-1} }{t_2^{\frac{q}{2}-1}} ~d\tau - \int_{0}^{t_1} \frac{f(x,t_1) \tau^{\frac{q}{2}-1} }{t_1^{\frac{q}{2}-1}} ~d\tau - \int_{t_1}^{t_2} \frac{f(x,\tau) \tau^{\frac{q}{2}-1}}{\tau^{\frac{q}{2}-1}} ~d\tau \right]\\
		& = q \left[\int_{t_1}^{t_2} \left(\frac{f(x,t_2) }{t_2^{\frac{q}{2}-1}} - \frac{f(x,\tau)}{\tau^{\frac{q}{2}-1}} \right) \tau^{\frac{q}{2}-1} ~d\tau + \int_0^{t_1}  \left(\frac{f(x,t_2) }{t_2^{\frac{q}{2}-1}} - \frac{f(x,t_1)}{t_1^{\frac{q}{2}-1}} \right) \tau^{\frac{q}{2}-1} ~d\tau \right] \geq 0,
	\end{align*}
	that is, $\mathcal{F}$ is increasing for $t \geq 0$. Analogously, $\mathcal{F}$ is decreasing for $t \leq 0$ and so,  \textnormal{(h$_2'$(i))} is satisfied.
\end{remark}

The proofs of Theorems \ref{main_result} and \ref{main_result-new}
rely on variational tools in combination with the Hardy-Littlewood-Sobolev inequality and the Nehari manifold. Indeed, let $I\colon\V\to\R$ be the energy functional of \eqref{problem}, then the Nehari manifold is defined as the set
\begin{align*}
	\mathcal{N}= \left\{u\in \V \setminus \{0\}:\; \langle  I'(u), u \rangle_{\mathcal{H}}=0\right \},
\end{align*}
where $\langle \cdot,\cdot\rangle_{\mathcal{H}}$ denotes the duality pairing between $\V$ and $\V^*$. It is clear that the set $\mathcal{N}$ is smaller than the whole space $\V$, but it contains all critical points of $I$ which are weak solutions of \eqref{problem}. We are looking for an element of $\mathcal{N}$ which realizes the infimum of $ \inf_{u\in \mathcal N}I(u)$. Such a function is a ground state solution of \eqref{problem}.

The paper is organized as follows. In Section \ref{sec_2} we present the main properties and embedding results for the Musielak-Orlicz Sobolev space $\V$ on the whole of $\R^N$ and we  recall the Hardy-Littlewood-Sobolev inequality which is used in our considerations. Section \ref{sec_3} is devoted to the proof of Theorem \ref{main_result} whereby the first part of this theorem is proved in Theorem \ref{exis:res}. Finally, in Section \ref{sec_4} we give the proof of Theorem \ref{main_result-new} without assuming the AR-condition.

\section{Preliminaries}\label{sec_2}

This section is devoted to recall the main preliminaries which are needed in the sequel, for example, the properties of the Musielak-Orlicz Sobolev space $\V$ and the Hardy-Littlewood-Sobolev inequality.

As usual, we denote by $L^r(\R)$ and  $L^r(\R^N)$ the classical Lebesgue spaces equipped with the norm $\|\cdot\|_r$ and for subsets $\Omega \subset \R^N$ we write $\|\cdot\|_{r,\Omega}$ for $1\leq r\leq \infty$. Furthermore, $W^{1,r}(\R^N)$ stands for the corresponding Sobolev space endowed with the $\|\cdot \|_{1,r}^r=\|\nabla \cdot \|_r^r+\|\cdot\|_r^r $ for any $1< r<\infty$. Consider the nonlinear function $\mathcal{H}\colon \R^N \times [0,\infty)\to [0,\infty)$ defined by
\begin{align*}
	\mathcal H(x,t)= t^p+a(x)t^q,
\end{align*}
where we suppose that hypotheses \textnormal{(h$_1$)} holds. Let $M(\R^N)$ be the space of all measurable functions $u\colon\R^N\to\R$. Then, Musielak-Orlicz Lebesgue space $L^\mathcal{H}(\R^N)$ is given by
\begin{align*}
	L^\mathcal{H}(\R^N)=\left \{u\in M(\R^N)\,:\,\varrho_{\mathcal{H}}(u):=\int_{\R^N} \mathcal{H}(x,|u|)\,\diff x<\infty \right \}
\end{align*}
equipped with the Luxemburg norm
\begin{align*}
	\|u\|_{\mathcal{H}} = \inf \left \{ \tau >0\,:\, \varrho_{\mathcal{H}}\left(\frac{u}{\tau}\right) \leq 1  \right \},
\end{align*}
where the modular function is given by
\begin{align*}
	\varrho_{\mathcal{H}}(u):=\int_{\R^N} \mathcal{H}(x,|u|)\,\diff x=\int_{\R^N} \big(|u|^{p}+a(x)|u|^q\big)\,\diff x.
\end{align*}
The corresponding Musielak-Orlicz Sobolev space $W^{1,\mathcal{H}}(\R^N)$ is defined by
\begin{align*}
	W^{1,\mathcal{H}}(\R^N)= \Big \{u \in L^\mathcal{H}(\R^N) \,:\, |\nabla u| \in L^{\mathcal{H}}(\R^N) \Big\}
\end{align*}
endowed with the norm
\begin{align*}
	\|u\|= \|\nabla u \|_{\mathcal{H}}+\|u\|_{\mathcal{H}},
\end{align*}
where $\|\nabla u\|_\mathcal{H}=\|\,|\nabla u|\,\|_{\mathcal{H}}$. Furthermore, we introduce the weighted space
\begin{align*}
	L^q_a(\R^N)=\left \{u\in M(\R^N)\,:\,\int_{\R^N} a(x) |u|^q \,\diff x< \infty \right \}
\end{align*}
with the seminorm
\begin{align*}
	\|u\|_{q,a} = \left(\int_{\R^N} a(x) |u|^q \,\diff x \right)^{\frac{1}{q}}.
\end{align*}

We know that $L^\mathcal{H}(\R^N)$ and $W^{1,\mathcal{H}}(\R^N)$ are separable reflexive Banach spaces, see Liu-Dai \cite[Theorem 2.7]{Liu-Dai-2020}. Moreover, $C_c^\infty(\R^N)$ is dense in $ W^{1,\mathcal{H}}(\R^N)$, see Harjulehto-H\"{a}st\"{o} \cite[Proposition 6.4.4]{Harjulehto-Hasto-2019} and Crespo-Blanco-Gasi\'{n}ski-Harjulehto-Winkert \cite[Theorems 2.24 and 2.28]{Crespo-Blanco-Gasinski-Harjulehto-Winkert-2022}.

The following relations between the norm $\|\,\cdot\,\|$ and the corresponding modular function can be found in Liu-Dai \cite[Proposition 2.6]{Liu-Dai-2020}.

\begin{proposition}\label{proposition_modular_properties}
	Let \textnormal{(h$_1$)} be satisfied, $u\in \V$, $c>0$ and
	\begin{align*}
		\rho(u)= \int_{\R^N}\big(|\nabla u|^p+|u|^p+a(x)(|\nabla u|^q+|u|^q)\big)\,\diff x=\|u\|_{1,p}^p+\|\nabla u\|_{q,a}^q+\|u\|_{q,a}^q.
	\end{align*}
	Then the following hold:
	\begin{enumerate}
		\item[\textnormal{(i)}]
			If $u\neq 0$, then $\|u\|=c$ if and only if $ \varrho(\frac{u}{c})=1$;
		\item[\textnormal{(ii)}]
			$\|u\|<1$ (resp.\,$>1$, $=1$) if and only if $ \varrho(u)<1$ (resp.\,$>1$, $=1$);
		\item[\textnormal{(iii)}]
			If $\|u\|<1$, then $\|u\|^q\leq \varrho(u)\leq\|u\|^p$;
		\item[\textnormal{(iv)}]
			If $\|u\|>1$, then $\|u\|^p\leq \varrho(u)\leq\|u\|^q$;
		\item[\textnormal{(v)}]
			$\|u\|\to 0$ if and only if $ \varrho(u)\to 0$;
		\item[\textnormal{(vi)}]
			$\|u\|\to \infty$ if and only if $ \varrho(u)\to \infty$.
	\end{enumerate}
\end{proposition}

The following embedding result can be found in Liu-Dai \cite[Theorem 2.7]{Liu-Dai-2020}.

\begin{proposition}\label{proposition_embeddings}
	Let \textnormal{(h$_1$)} be satisfied. Then the embedding $\V \hookrightarrow L^r(\R^N)$ is continuous for all $r \in [p,p^*]$.
\end{proposition}

One main tool in our treatment is the famous Hardy-Littlewood-Sobolev inequality, see, for example, Lieb-Loss \cite[Theorem 4.3]{Lieb-Loss-2001}.

\begin{proposition}[Hardy-Littlewood-Sobolev inequality]\label{HLS}

	Let $s, r>1$ and $0<\mu<N $ with $\frac{1}{s}+\frac{\mu}{N}+\frac{1}{r}=2$ and let $g \in L^s(\R^N)$ and $h \in L^r(\R^N)$. Then there exists a sharp constant $C(N,\mu,s)$, independent of $g$ and $h$, such that
	\begin{align}\label{HLSineq}
		\int_{\R^N}\int_{\R^N} \frac{g(x)h(y)}{|x-y|^{\mu}}\,\diff x\,\diff y \leq C(N,\mu,s)\|g\|_{s}\|h\|_{r}.
	\end{align}
	If $s =r =\frac{2N}{2N-\mu}$, then
	\begin{align*}
		C(N,\mu,s)= C(N,\mu)= \pi^{\frac{\mu}{2}} \frac{\Gamma\left(\frac{N}{2}-\frac{\mu}{2}\right)}{\Gamma\left(N-\frac{\mu}{2}\right)} \left\{ \frac{\Gamma\left(\frac{N}{2}\right)}{\Gamma(N)} \right\}^{-1+\frac{\mu}{N}}.
	\end{align*}
	In this case, we have an equality in \eqref{HLSineq} if and only if $g\equiv (constant)h$ and
	\begin{align*}
		h(x)= A\left(\gamma^2+ |x-a|^2\right)^{\frac{-(2N-\mu)}{2}}
	\end{align*}
	for some $A \in \mathbb C$, $0 \neq \gamma \in \mathbb R$ and $a \in \R^N$.
 \end{proposition}

Let $(X,\|\cdot\|_X)$ be a Banach space, $(X^*,\|\cdot\|_{X^*})$ its topological dual space and  $\varphi\in C^1(X)$. We say that $\{u_n\}_{n\in\mathbb N}\subset X$ is a Palais-Smale sequence at level $c\in\R$ (\textnormal{(PS)$_c$}-sequence for short) for $\varphi$ if
\begin{align*}
	\varphi(u_n)\to c \quad \text{and}\quad \varphi'(u_n)\to 0 \quad\text{in }X^*\quad\text{as }n\to\infty.
\end{align*}
We say that $\varphi$ satisfies the Palais-Smale condition at level $c\in\R$ (\textnormal{(PS)$_c$}-condition for short) if any \textnormal{(PS)$_c$}-sequence $\{u_n\}_{n\in\N}$ admits a convergent subsequence in $X$. If this condition holds at every level $c\in\R$, then we say that $\varphi$ satisfies the Palais-Smale condition (the \textnormal{(PS)}-condition for short). Moreover, $\varphi$ satisfies the Cerami-condition at level $c\in\R$ (\textnormal{(C)}$_c$-condition for short), if every sequence $\{u_n\}_{n\in\N}\subseteq X$ such that
\begin{align*}
		\varphi(u_n)\to c \quad \text{and}\quad \left(1+\|u_n\|_X\right)\varphi'(u_n)\to 0 \quad \text{ in }X^* 
		\quad\text{as }n\to \infty,
\end{align*}
admits a strongly convergent subsequence. If this condition holds at every level $c\in\R$, then we say that $\varphi$ satisfies the Cerami condition (the \textnormal{(C)}-condition for short).

\section{Existence of a ground state solution with AR-condition}\label{sec_3}

In this section we give the proof of Theorem \ref{main_result} under the hypotheses \textnormal{(h$_1$)}--\textnormal{(h$_4$)}. The energy functional $I\colon\V\to\R$ associated to \eqref{problem} is given by
\begin{align*}
	I(u)
	&= \frac{\|u\|_{1,p}^p}{p}+ \frac{\|\nabla u\|_{q,a}^q+\|u\|_{q,a}^q}{q}-\frac{1}{2}\int_{\R^N}\left( \int_{\R^N} \frac{F(y,u)}{|x-y|^\mu}\,\diff y\right)F(x,u)\,\diff x,
\end{align*}
which is clearly $C^1$ with derivative
\begin{align*} 
	\langle I'(u), v \rangle_{\mathcal{H}} 
	&= \int_{\R^N}\big(|\nabla u|^{p-2}\nabla u+ a(x)|\nabla u|^{q-2}\nabla u\big)\cdot \nabla v\,\diff x + \int_{\R^N}\big(| u|^{p-2} u+ a(x)| u|^{q-2} u\big) v\,\diff x\\
	&\quad -\int_{\R^N}\left( \int_{\R^N} \frac{F(y,u)}{|x-y|^\mu}\,\diff y \right)f(x,u) v\,\diff x\quad \text{for all }u,v \in \V,
\end{align*}
where $\langle \cdot, \cdot\rangle_{\mathcal{H}}$  denotes the duality pairing between $\V$ and its dual
space $\V^*$. Clearly, the critical points of $I$ correspond to the weak solutions of problem \eqref{problem}. In order to establish the existence of a weak solution, we consider the Mountain pass level
\begin{align}\label{MP-level}
	b:= \inf_{\gamma \in \Gamma}\sup_{t\in [0,1]}I(\gamma(t)),
\end{align}
where 
\begin{align*}
	\Gamma=\left\{\gamma \in C\left([0,1],\V\right) \,:\, \gamma(0)=0,\, I(\gamma(1))<0\right\}.
\end{align*}
By using Proposition \ref{HLS} we can estimate the Choquard term by
\begin{align}\label{hls-1}
	\int_{\R^N}\left( \int_{\R^N} \frac{F(y,u)}{|x-y|^\mu}\,\diff  y\right)F(x,u)\,\diff x \leq C(N,\mu) \|F(\cdot, u)\|_{\frac{2N}{2N-\mu}}^2.
\end{align}
\vspace{0.1cm}

\begin{lemma}\label{lemma-positive-b}
	Let hypotheses \textnormal{(h$_1$)}--\textnormal{(h$_2$)} be satisfied. It holds $b>0$.
\end{lemma}

\begin{proof}
	From \eqref{hls-1} and \eqref{new:nonl:growth}, along with Proposition \ref{proposition_embeddings} and \eqref{ineq:r1r2}, we have
	\begin{align}\label{hls-new16}
		\begin{split}
			\int_{\R^N}\left( \int_{\R^N} \frac{F(y,u)}{|x-y|^\mu}\,\diff y\right)F(x,u)\,\diff x 
			&\leq C_1  \left(\int_{\R^N}\left(|u|^{\frac{2Nr_1}{2N-\mu}}+ |u|^{\frac{2Nr_2}{2N-\mu}}\right)\,\diff x\right)^{\frac{2N-\mu}{N}}\\
			& \leq C_2 (\|u\|_{1, p}^{2r_1}+\|u\|_{1, p}^{2r_2})\\
			&\leq C_3 \left(\left(\rho(u)\right)^{\frac{2r_1}{p}}+\left(\rho(u)\right)^{\frac{2r_2}{p}}\right)
		\end{split}
	\end{align}
	with some $C_1, C_2, C_3>0$. Applying \eqref{hls-new16} we get
	\begin{align*}
		I(u) \geq  \frac{1}{q}\rho(u) -  \frac12\int_{\R^N}\left( \int_{\R^N} \frac{F(y,u)}{|x-y|^\mu}\,\diff y\right)F(x,u)\,\diff x
		\geq \frac{1}{q}\rho(u) - \frac{C_3}{2} \left(\rho(u))^{\frac{2r_1}{p}}-(\rho(u))^{\frac{2r_2}{p}}\right).
	\end{align*}
	Since $2r_2\geq 2r_1>p$ by \eqref{growth:exponent:cond}, we may choose $\delta>0$ such that if $\rho(u)\leq \delta$ then
	\begin{align*}
		I(u)\geq \frac{\rho(u)}{2q}.
	\end{align*}
	In particular, if $\gamma \in \Gamma$, then we have $\rho(\gamma(0))=0<\delta< \rho(\gamma(1))$ since $I(\gamma(1))<0$ and $\rho(u)\leq \delta$ implies $I(u)>0$. Therefore, by the intermediate value theorem, there exists $ \tau_0\in(0,1)$ such that $\rho(\gamma(\tau_0))=\delta$. This gives us
	\begin{align*}
		\frac{\delta}{2q} \leq I(\gamma(\tau_0)) \leq \sup_{t\in[0,1]}I(\gamma(t)).
	\end{align*}
	As $\gamma\in \Gamma$ was arbitrary, we get $b\geq \frac{\delta}{2q}>0$.
\end{proof}

\begin{lemma}\label{M-P}
	Let hypotheses \textnormal{(h$_1$)}--\textnormal{(h$_2$)} be satisfied.
	\begin{enumerate}
		\item[\textnormal{(i)}] 
			There exist $\varrho$, $\eta>0$ such that $I(u)\geq \eta$ for all $u\in \V$ with $\|u\|=\varrho$.
		\item[\textnormal{(ii)}] 
			There exists $e\in \V$ with $\|e\|>\varrho$ such that $I(e)<0$.
	\end{enumerate}
\end{lemma}

\begin{proof}
	\textnormal{(i)} By \eqref{hls-1} and \textnormal{(h$_2$)}, we can write 
	\begin{align*}
		I(u) & \geq \frac{\|u\|_{1,p}^p}{p}+ \frac{\|\nabla u\|_{q,a}^q+\|u\|_{q,a}^q}{q}-\frac{1}{2}C(N,\mu) \|F(\cdot,u)\|_{\frac{2N}{2N-\mu}}^2\\
		& \geq \frac{\|u\|_{1,p}^p}{p}+ \frac{\|\nabla u\|_{q,a}^q+\|u\|_{q,a}^q}{q}- C_4 \left(\int_{\R^N}\left(|u|^{\frac{2Nr_1}{2N-\mu}}+ |u|^{\frac{2Nr_2}{2N-\mu}}\right)\,\diff x\right)^{\frac{2N-\mu}{N}},
	\end{align*}
	where $C_4>0$ is constant independent of $u$. Using again Proposition \ref{proposition_embeddings} and \eqref{ineq:r1r2}, we obtain
	\begin{align*}
		I(u) & \geq \frac{1}{q} \rho(u)- C_5 (\|u\|_{1, p}^{2r_1}+\|u\|_{1, p}^{2r_2})
		\geq \frac{1}{q}\|u\|_{1,p}^p- C_5 (\|u\|_{1,p}^{2r_1}+\|u\|_{1,p}^{2r_2})
	\end{align*}
	for some $C_5>0$. Since $2r_2\geq 2r_1 >p$, we can choose $\varrho>0$ sufficiently small such that $I(u) \geq \eta $ provided $\|u\|=\rho$ for some $\eta >0$. 
	
	\textnormal{(ii)}
	In order to prove the second part, let us fix $u_0\in \V\setminus\{0\}$ with $u_0\geq 0$ and define
	\begin{align*}
		 g_x(t)= F\left(x,\frac{tu_0}{\|u_0\|}\right) \quad \text{for } t>0 \ \text{and} \ x \in \Omega.
	\end{align*}
	From \textnormal{(h$_2$)} it follows that
	\begin{align*}
		\frac{ g_x'(t)}{ g_x(t)}\geq \frac{\theta }{2t} \quad \text{for } t>0. 
	\end{align*}
	Integrating this over $[1,s\|u_0\|]$ with $s>\frac{1}{\|u_0\|}$, we easily get
	\begin{align*}
	 g_x(s\|u_0\|)\geq g_x(1)(s\|u_0\|)^{\frac{\theta}{2}},
	\end{align*}
	that is,
	\begin{align*}
		 F(x,s u_0)  \geq F\left(x,\frac{u_0}{\|u_0\|}\right)(s\|u_0\|)^{\frac{\theta}{2}}.
	\end{align*}
	Using this, we can write
	\begin{align*}
		I(su_0)&\leq \frac{s^p}{p}\|u_0\|_{1,p}^p +\frac{s^q}{q}\left(\|\nabla u_0\|_{q,a}^q+\|u_0\|_{q,a}^q\right)\\
		&\quad-\frac{s^\theta\|u_0\|^\theta}{2} \int_{\R^N}\left( \int_{\R^N} \frac{F\left(y,\frac{u_0}{\|u_0\|}\right)}{|x-y|^\mu}\,\diff y\right)F\left(x,\frac{u_0}{\|u_0\|}\right)\,\diff x \\
		&= C_6s^p+C_7s^q-C_7s^\theta,
	\end{align*}
	where $C_6, C_7, C_8$ are positive constants and $s>\frac{1}{\|u_0\|}$. Therefore we can choose $s>\frac{1}{\|u_0\|}$ large enough such that $e= su_0$ with $I(e)<0$ and $\|e\|>\rho$ since $\theta >q>p$.
\end{proof}

By the Mountain Pass theorem without \textnormal{(PS)}-condition, see Chabrowski \cite[Theorem 5.4.1]{Chabrowski-1997}, there exists a \textnormal{(PS)$_b$}-sequence $\{u_n\}_{n\in\N}\subset \V$ of $I$, that is,
\begin{align}\label{ps-seq}
	I(u_n) \to b \quad \text{and}\quad I'(u_n) \to 0 
	\quad \text{in }  \V^*,
\end{align}
where $b$ is defined in \eqref{MP-level}.

\begin{lemma}\label{PS-bdd}
	Let hypotheses \textnormal{(h$_1$)}--\textnormal{(h$_2$)} be satisfied. Then, the \textnormal{(PS)$_b$}-sequence $\{u_n\}_{n\in\N}\subset \V$ of $I$ is bounded.
\end{lemma}

\begin{proof}
	From \eqref{ps-seq}, we get 
	\begin{align}\label{mp-level-2}
		I(u_n)-\frac{\langle I'(u_n), u_n \rangle_{\mathcal{H}}}{\theta} \leq b+1+\|u_n\|,
	\end{align}
	for $n\in\N$ large enough. By \textnormal{(h$_2$)}  and Proposition \ref{proposition_modular_properties} we have for $\|u_n\|\geq 1$ that
	\begin{align*}
		I(u_n)-\frac{\langle I'(u_n), u_n \rangle_{\mathcal{H}}}{\theta}
		&= \left(\frac{1}{p}-\frac{1}{\theta}\right)\|u_n\|_{1,p}^p + \left(\frac{1}{q}-\frac{1}{\theta}\right) \left(\|\nabla u_n\|_{q,a}^q+\|u_n\|_{q,a}^q\right)\\
		&\quad - \int_{\R^N}\left( \int_{\R^N} \frac{F(y, u_n)}{|x-y|^\mu}\,\diff y\right) \left(\frac{F(x, u_n)}{2}-\frac{u_n f(x,u_n)}{\theta}\right)\,\diff x\\
		& \geq \left(\frac{1}{p}-\frac{1}{\theta}\right)\|u_n\|_{1,p}^p + \left(\frac{1}{q}-\frac{1}{\theta}\right) \left(\|\nabla u_n\|_{q,a}^q+\|u_n\|_{q,a}^q\right) \\
		&\geq \left(\frac{1}{q}-\frac{1}{\theta}\right) \rho(u_n) \geq \left(\frac{1}{q}-\frac{1}{\theta}\right) \|u_n\|^p.
	\end{align*}
	This along with \eqref{mp-level-2} gives the boundedness of $\{u_n\}_{n\in\N}$ in $\V$.
\end{proof}

\begin{lemma}\label{lionl}
	Let hypotheses \textnormal{(h$_1$)}--\textnormal{(h$_2$)} be satisfied. Then, there exist $r,\beta>0$ and a sequence $\{y_n\}_{n\in\N}\subset \R^N$ such that
	\begin{align*}
		\liminf_{n\to \infty}\int_{B_r(y_n)}|u_n(x)|^p \,\diff x \geq \beta.
	\end{align*}
\end{lemma}

\begin{proof}
	Suppose the assertion is not true. Then, by Lions' lemma \cite[Lemma I.1]{Lions-1984}, one has
	\begin{align*}
		u_n\to 0 \quad \text{in } L^\alpha(\R^N)\quad \text{for any }
		\alpha \in (p,p^*).
	\end{align*}
	From \eqref{hls-1} and \textnormal{(h$_2$)}, we know that
	\begin{align*}
		\int_{\R^N}\left( \int_{\R^N} \frac{F(y, u_n)}{|x-y|^\mu}\,\diff y\right)F(x, u_n)\,\diff x \leq  C_0 \left(\int_{\R^N}\left(|u_n|^{\frac{2Nr_1}{2N-\mu}}+ |u_n|^{\frac{2Nr_2}{2N-\mu}}\right)\,\diff x\right)^{\frac{2N-\mu}{N}},
	\end{align*}
	with a constant $C_0>0$. Due to \eqref{ineq:r1r2} it follows that
	\begin{align}\label{limit-1}
		\lim_{n\to \infty}\int_{\R^N}\left( \int_{\R^N} \frac{F(y, u_n)}{|x-y|^\mu}\,\diff y\right)F(x, u_n)\,\diff x =0
	\end{align}
	and similarly,
	\begin{align}\label{limit-2}
		\lim_{n\to \infty}\int_{\R^N}\left( \int_{\R^N} \frac{F(y, u_n)}{|x-y|^\mu}\,\diff y\right)u_n f(x, u_n)\,\diff x =0.
	\end{align}
	Using \eqref{limit-2} in $\lim_{n\to \infty} I'(u_n)=0$, we easily get
	\begin{align}\label{limit-3}
		\lim_{n\to \infty}\left(\|u_n\|_{1,p}^p + \|\nabla u_n\|_{q,a}^q +  \|u_n\|_{q,a}^q\right)=0.
	\end{align}
	On the other hand, from \eqref{limit-1} and \eqref{limit-3} we get $0=\lim\limits_{n\to \infty}I(u_n)=  b>0$ which is a contradiction.
\end{proof}

Now, we define a sequence $v_n(\,\cdot\,)=u_n(\cdot+y_n)$. Then $\varrho_{\mathcal H}(u_n)=\varrho_{\mathcal H}(v_n)$, so $\{v_n\}_{n\in\N}$ remains bounded in $\V$, see Lemma \ref{PS-bdd}. Moreover, by translation invariance of $I$ and $I'$ due to  \textnormal{(h$_3$)}, implies
\begin{align}\label{new-PS}
	I(v_n)\to b \quad\text{and}\quad I'(v_n)\to 0.
\end{align}
Thus, up to a subsequence, there exists $v\in \V$ such that
\begin{equation}\label{soluzione}
v_n\rightharpoonup v\mbox{ in }\V,\quad v_n\to v\mbox{ in }L^s_{\loc}(\R^N)\mbox{ for any }s\in[1,p^*),
\end{equation}
and also by Lemma \ref{lionl}
\begin{align*}
	\int_{B_r(0)} |v_n(x)|^p\,\diff x \geq \frac{\beta}{2}.
\end{align*} 
From this, it is clear that $v\neq 0$. We shall now use the \textnormal{(PS)$_b$}-sequence $\{v_n\}_{n\in\N}$ for our future purposes.

\begin{proposition}\label{choq-conv}
	Let hypotheses \textnormal{(h$_1$)}--\textnormal{(h$_3$)} be satisfied. For any $\varphi\in C_0^\infty(\R^N)$ we have, up to a subsequence,
	\begin{align*}
		\lim_{n\to \infty}\int_{\R^N}\left( \int_{\R^N} \frac{F(y, v_n)}{|x-y|^\mu}\,\diff y\right)\varphi f(x, v_n)\,\diff x = \int_{\R^N}\left( \int_{\R^N} \frac{F(y,v)}{|x-y|^\mu}\,\diff y\right)\varphi f(x, v)\,\diff x.
	\end{align*}
\end{proposition}

\begin{proof}
	By the growth conditions in \eqref{growth:exponent:cond} and \eqref{new:nonl:growth} along with the fact that $v_n$ is bounded in $\V$ gives the boundedness of $F(\cdot, v_n)$ in $L^{\frac{2N}{2N-\mu}}(\R^N)$. In addition, the pointwise convergence of $v_n$ to $v$ and the continuity of $F$ imply that $F(x, v_n)\to F(x,v)$ pointwise a.\,e.\, in $\R$. We define the convolution operator $K\colon L^{\frac{2N}{2N-\mu}}(\R^N) \to L^{\frac{2N}{\mu}}(\R^N)$ by
	\begin{align*}
		K(w)(x):= \frac{1}{|x|^\mu}\ast w(x).
	\end{align*}
	From the Hardy-Littlewood-Sobolev inequality stated in Proposition \ref{HLS}, we obtain that $K$ is a linear and bounded operator. Hence, up to a subsequence, $\{K(F(\cdot, v_n))\}_{n\in\N}$ is bounded in $L^{\frac{2N}{\mu}}(\R^N)$,
	\begin{align*}
		K(F(x,v_n)) \to K(F(x,v))\quad \text{a.\,e.\,in } \R^N
	\end{align*}
	and 
	\begin{align*}
		\int_{\R^N} \int_{\R^N}\frac{F(y, v_n)}{|x-y|^\mu} \psi(x) \,\diff y\,\diff x \to \int_{\R^N} \int_{\R^N}\frac{F(y,v)}{|x-y|^\mu} \psi(x) \,\diff y\,\diff x \quad \text{for every} \ \psi \in L^{\frac{2N}{2N-\mu}}(\R^N).
	\end{align*}
	In particular, for every $\phi \in C_c^\infty(\R^N)$, we have 
	\begin{align}\label{conv:est-1}
		\int_{\R^N} \int_{\R^N} \frac{F(y, v_n)}{|x-y|^\mu} f(x,v) \phi(x) \,\diff y \,\diff x \to \int_{\R^N} \int_{\R^N}\frac{F(y,v)}{|x-y|^\mu} f(x,v) \phi(x) \,\diff y \,\diff x.
	\end{align}
	Now, we claim that for every $\phi \in C_c^\infty(\R^N)$,
	\begin{align}\label{conv:est-2}
		\int_{\R^N} \left(\int_{\R^N} \frac{F(y,v_n)}{|x-y|^\mu} \,\diff y \right) \left( f(x,v_n)- f(x,v)\right) \phi(x) \,\diff x \to 0.
	\end{align}
	Since $\{K(F(\cdot, v_n))\}_{n\in\N}$ is uniformly bounded in $L^\frac{2N}{\mu}(\R^N)$, by H\"older's inequality, it is enough to show that
	\begin{align}\label{conv:est-3}
		\|\left(f(\cdot, v_n)- f(\cdot, v)\right) \phi\|_{\frac{2N}{2N-\mu},\supp(\phi)} \to 0.
	\end{align}
	Using \eqref{new:nonl:growth} and Young's inequality we get
	\begin{align*}
		\left[f(\cdot, v_n) \phi \right]^\frac{2N}{2N-\mu} & \leq C \left(|v_n|^{\frac{2N(r_1-1)}{2N-\mu}} + |v_n|^{\frac{2N(r_2-1)}{2N-\mu}}  \right) \phi^\frac{2N}{2N-\mu} \\
		& \leq C_1 \left( |v_n|^{\frac{2N r_1}{2N-\mu}} + |v_n|^{\frac{2N r_2}{2N-\mu}} \right) + C_2(r_1, r_2, \|\phi\|_{\infty}) \in L^1(\supp(\phi))
	\end{align*}
	for some $C_1, C_2>0$. Then, from Lebesgue's dominated convergence theorem, we obtain the required assertion in \eqref{conv:est-3} and so in \eqref{conv:est-2}. Finally, combining the estimates in \eqref{conv:est-1} and \eqref{conv:est-2}, it is easy to verify that
	\begin{align*}
		\lim_{n\to \infty}\int_{\R^N}\left( \int_{\R^N} \frac{F(y,v_n)}{|x-y|^\mu}\,\diff y\right)\varphi f(x, v_n)\,\diff x = \int_{\R^N}\left( \int_{\R^N} \frac{F(y,v)}{|x-y|^\mu}\,\diff y\right)\varphi f(x, v) \,\diff x.
	\end{align*}
\end{proof}

\begin{proposition}\label{grad-con}
	Let hypotheses \textnormal{(h$_1$)}--\textnormal{(h$_3$)} be satisfied. For a subsequence of $\{v_n\}_{n\in\N}$, we have
	\begin{align*}
		\nabla v_n \to \nabla v \quad\text{pointwise a.\,e.\,in } \R^N.
	\end{align*}
	Consequently, it holds,
	\begin{align*}
		|\nabla v_n|^{p-2}\nabla v_n &\rightharpoonup |\nabla v|^{p-2}\nabla v \quad \text{in } [L^{\frac{p}{p-1}}(\R^N)]^N;\\
		|\nabla v_n|^{q-2}\nabla v_n &\rightharpoonup |\nabla v|^{q-2}\nabla v\quad \text{in } [L^{\frac{q}{q-1}}_a(\R^N)]^N.
	\end{align*}
\end{proposition}

\begin{proof}
	We know that
	\begin{align}\label{known:conv}
		v_n \rightharpoonup v \quad \text{in }  \V,
		\quad
		v_n \to v \quad \text{in } L^s_{\loc}(\R^N)
		\quad\text{and}\quad
		v_n \to v \quad \text{a.\,e.\,in } \R^N
	\end{align}
	for $s \in [1, p^*]$. Let $\psi \in C_c^\infty(\R^N)$, $\psi \geq 0$, $\psi =1$ in $B_R \subset \supp(\psi)$ with $R>0$. Taking $\phi= (v_n -v) \psi$ as test function in \eqref{new-PS} leads to
	\begin{align}\label{conv:est-0}
		\begin{split}
			&\lim_{n \to \infty}  \int_{\R^N} \left( |\nabla v_n|^{p-2} \nabla v_n - |\nabla v|^{p-2} \nabla v + a(x) \left(|\nabla v_n|^{q-2} \nabla v_n - |\nabla v|^{q-2} \nabla v\right)\right) \cdot (\nabla (v_n-v)) \psi \,\diff x\\
			& = -  \lim_{n \to \infty} \int_{\R^N} \bigg( |\nabla v_n|^{p-2} \nabla v_n - |\nabla v|^{p-2} \nabla v\\
			&\qquad\qquad \qquad\qquad + a(x) \left(|\nabla v_n|^{q-2} \nabla v_n - |\nabla v|^{q-2} \nabla v\right)\bigg) \cdot \nabla \psi \ (v_n-v) \,\diff x\\
			& \quad -  \lim_{n \to \infty}  \int_{\R^N} \left( |\nabla v|^{p-2} \nabla v + a(x)|\nabla v|^{q-2} \nabla v\right)  \cdot \nabla (v_n-v) \psi \,\diff x\\ 
			&\quad -  \lim_{n \to \infty}  \int_{\R^N} \left(|v_n|^p + a(x) |v_n|^q\right) (v_n-v) \psi \,\diff x\\
			& \quad +  \lim_{n \to \infty}  \int_{\R^N} \left( \int_{\R^N} \frac{F(y, v_n)}{|x-y|^\mu}\,\diff y\right)f(x, v_n) (v_n -v) \varphi\,\diff x.
		\end{split}
	\end{align}
	By H\"older's inequality, we observe that 
	\begin{align*}
		[L^{\mathcal{H}}(\R^N)]^N \ni h \longmapsto \int_{\R^N} \left( |\nabla v|^{p-2} + a(x)|\nabla v|^{q-2} \right) \nabla v \cdot h \,\diff x
	\end{align*}
	and
	\begin{align*}
		L^{\mathcal{H}}(\R^N) \ni g \longmapsto \int_{\R^N} \left( |\nabla v|^{p-2} + a(x)|\nabla v|^{q-2} \right)  v \cdot g \,\diff x
	\end{align*}
	are bounded linear functionals. Now, by using \eqref{known:conv} and  $\psi \in C_c^\infty(\R^N)$, it is clear that \begin{align}\label{conv-est-1}
		\begin{split}
			\lim_{n \to \infty}  \int_{\R^N} 
			&\Big(\left( |\nabla v|^{p-2} \nabla v + a(x) |\nabla v|^{q-2} \nabla v\right)  \cdot \nabla (v_n-v)\\
			&\quad  +  (|v_n|^p + a(x) |v_n|^q) (v_n-v)\Big) \psi \,\diff x =0,
		\end{split}
	\end{align}
	and
	\begin{align}\label{conv-est-2}
		\begin{split}
			\lim_{n \to \infty} \int_{\R^N} &\Big(\left( |\nabla v_n|^{p-2} \nabla v_n -  |\nabla v|^{p-2} \nabla v\right)\\ 
			&\qquad + a(x) \left(|\nabla v_n|^{q-2} \nabla v_n - |\nabla v|^{q-2} \nabla v\right)\Big) \cdot \nabla \psi \ (v_n-v) \,\diff x=0.
		\end{split}
	\end{align}
Following the arguments of Proposition \ref{choq-conv} and using \eqref{known:conv} as well as \eqref{new:nonl:growth}, we obtain
	\begin{align}\label{new-4}
		K(F(x,v_n)) := \int_{\R^N}  \frac{F(y, v_n)}{|x-y|^\mu}\,\diff y
		\quad \text{is uniformly bounded w.\,r.\,t.\,} n\in\mathbb N  \text{ in }
		L^{\frac{2N}{\mu}}(\R^N),
	\end{align}
	and by using Young's inequality, we have
	\begin{align}\label{new-3}
		\begin{split}
			&| K(F(\cdot, v_n)) f(\cdot, v_n) (v_n-v) \psi|\\
			&\leq C_0 \left( |K(F(\cdot, v_n))|^\frac{2N}{\mu} + |f(\cdot, v_n) (v_n-v) \psi|^{\frac{2N}{2N-\mu}} \right)\\
			& \leq  C_1 \left( |K(F(\cdot, v_n))|^\frac{2N}{\mu} + \left(|v_n|^{\frac{2N(r_1-1))}{2N-\mu}} |(v_n-v) \psi|^\frac{2N}{2N-\mu} + |v_n|^{\frac{2N(r_2-1))}{2N-\mu}} |(v_n-v) \psi|^\frac{2N}{2N-\mu} \right) \right)\\
			& \leq C_2 ( |K(F(\cdot, v_n))|^\frac{2N}{\mu} + |v_n|^{\frac{2r_1N}{2N-\mu}} + |v_n|^{\frac{2r_2N}{2N-\mu}} + |v_n-v|^{\frac{2r_1N}{2N-\mu}} +|v_n-v|^{\frac{2r_2N}{2N-\mu}})  \in L^1(\supp(\psi))
		\end{split}
	\end{align}
	due to \eqref{ineq:r1r2} whereby $C_0, C_1, C_2$ are positive constants. Combining the above facts and using Lebesgue's dominated convergence theorem, we get
	\begin{align}\label{conv-est-3}
		\lim_{n \to \infty}  \int_{\R^N} \left( \int_{\R^N} \frac{F(y,v_n)}{|x-y|^\mu}\,\diff y\right)f(x, v_n) (v_n -v) \varphi \,\diff x =0.
	\end{align}
	Now, using the convergence results of \eqref{conv-est-1}-\eqref{conv-est-3} in \eqref{conv:est-0}, it follows that
	\begin{align*}
		\lim_{n \to \infty}  \int_{\R^N} \left[\left( |\nabla v_n|^{p-2} \nabla v_n - |\nabla v|^{p-2} \nabla v\right) + a(x) \left(|\nabla v_n|^{q-2} \nabla v_n - |\nabla v|^{q-2} \nabla v\right)\right] \cdot \nabla (v_n-v)\ \psi \,\diff x =0.
	\end{align*}
	On the last expression we can apply Simon's inequalities (see Simon \cite[formula (2.2)]{Simon-1978}) and use the fact that $\psi =1$ in $B_R$. This gives
	\begin{align*}
		\lim_{n \to \infty}  \int_{B_R} |\nabla v_n - \nabla v|^p\,\diff x =0,
	\end{align*}
	and since the choice of cut-off function $\psi$ with $B_R \subset \supp(\psi)$, $R>0$ is arbitrary,
	\begin{align*}
		\nabla v_n \to \nabla v \quad \text{pointwise a.\,e.\,in }  \R^N.
	\end{align*}
However, this says that
	\begin{align*}
			|v_n|^{p-2} v_n \to |v|^{p-2} v
			\quad \text{pointwise a.\,e.\,in }  \R^N.
	\end{align*}
Since $\{|v_n|^{p-2} v_n \}_{n\in\N}$ is bounded in $[L^{\frac{p}{p-1}}(\R^N)]^N$,  we conclude that
	\begin{align*}
		|\nabla v_n|^{p-2}\nabla v_n \rightharpoonup |\nabla v|^{p-2}\nabla v
		\quad \text{in } [L^{\frac{p}{p-1}}(\R^N)]^N.
	\end{align*}
In a similar way, we can establish that
	\begin{align*}
		|\nabla v_n|^{q-2}\nabla v_n \rightharpoonup |\nabla v|^{q-2}\nabla v
		\quad \text{in } [L^{\frac{q}{q-1}}_a(\R^N)]^N.
	\end{align*}
\end{proof}

Now we can prove that problem \eqref{problem} has a nontrivial weak solution which shows the first part of Theorem \ref{main_result}.

\begin{theorem}\label{exis:res}
	Let hypotheses \textnormal{(h$_1$)}--\textnormal{(h$_3$)} be satisfied. Then the element $v\in V$ set in \eqref{soluzione} is a critical point of the functional $I$, and so a weak solution for problem \eqref{problem}. Moreover, $I(v) \leq b$.
\end{theorem}

\begin{proof}
	The proof is provided by showing
	\begin{align}\label{th-1}
		\lim_{n\to \infty}\left\langle I'(v_n),\varphi\right\rangle_{\mathcal{H}}= \left\langle I'(v),\varphi\right\rangle_{\mathcal{H}}	\quad \text{for all } \varphi\in C_c^\infty(\R^N).
	\end{align}
	Recall that
	\begin{align*}
			\langle I'(v_n),\varphi\rangle_{\mathcal{H}}
			&= \int_{\R^N}\left(|\nabla v_n|^{p-2}\nabla v_n+ a(x)|\nabla v_n|^{q-2}\nabla v_n\right)\cdot \nabla \varphi\,\diff x  + \int_{\R^N}\left(| v_n|^{p-2} v_n+ a(x)| v_n|^{q-2} v_n\right) \varphi\,\diff x\\
			& \qquad-  \int_{\R^N}\left( \int_{\R^N} \frac{F(y,v_n)}{|x-y|^\mu}\,\diff y\right)f(x, v_n) \varphi\,\diff x.
	\end{align*}
	Applying Propositions \ref{grad-con} and \ref{choq-conv} we easily derive \eqref{th-1}. Now, using \eqref{new-PS} and  the density of $C_c^\infty(\R^N)$ in $\V$ given in \cite[Proposition 6.4.4]{Harjulehto-Hasto-2019}, we obtain $I'(v)=0$.
	
	Let us now prove that $I(v) \leq b$. From Proposition \ref{grad-con}, \eqref{AR-cond}, Fatou's lemma and the fact that $v_n \to v$ in $L^{q}(B_\eta(0))$ for any $\eta >0$, we derive
	\begin{align*}
		\lim\inf_{ n \to \infty} \left(I(v_n) - \frac{1}{\theta} I'(v_n) v_n \right) &= \liminf_{n\to \infty} \left(\frac{(\theta-p)}{p\theta}\|v_n\|_{1,p}^p + \frac{(\theta-q)}{q\theta} \left(\|\nabla v_n\|_{q,a}^q+\|v_n\|_{q,a}^q \right)\right)\\
		& \quad + \liminf_{n\to \infty}  \frac{1}{2\theta} \int_{\R^N}\left( \int_{\R^N} \frac{F(y,v_n)}{|x-y|^\mu}\,\diff y\right) \left(2 v_n f(x,v_n) - \theta F(x,v_n)\right) \,\diff x \\
		& \geq \liminf_{n\to \infty} \left(\frac{(\theta-p)}{p\theta}\|v_n\|_{1,p}^p +  \frac{(\theta-q)}{q\theta} \left(\|\nabla v_n\|_{q,a}^q+\|v_n\|_{q,a}^q \right)\right) \\
		& \quad  + \liminf_{n\to \infty} \frac{1}{2\theta} \int_{B_{\eta}(0)}\left( \int_{B_{\eta}(0)} \frac{F(y,v_n)}{|x-y|^\mu}\,\diff y\right) \left(2 v_n f(x,v_n) - \theta F(x,v_n)\right) \,\diff x\\
		& = \frac{(\theta-p)}{p\theta}\|v\|_{1,p}^p  + \frac{(\theta-q)}{q\theta} \left(\|\nabla v\|_{q,a}^q+\|v\|_{q,a}^q \right) \\
		& \quad + \frac{1}{2\theta} \int_{B_{\eta}(0)}\left( \int_{B_{\eta}(0)} \frac{F(y,v)}{|x-y|^\mu}\,\diff y\right) \left(2 v f(x,v) - \theta F(x,v)\right) \,\diff x.
	\end{align*}
	Since the left hand side of above inequality tends to $b$ as $n \to \infty$ and $\eta$ is arbitrary, we obtain
	\begin{align*}
		b & \geq \frac{(\theta-p)}{p\theta}\|v\|_{1,p}^p + \frac{(\theta-q)}{q\theta} \left(\|\nabla v\|_{q,a}^q+\|v\|_{q,a}^q \right) + \frac{1}{2\theta} \int_{\R^N}\left( \int_{\R^N} \frac{F(y,v)}{|x-y|^\mu}\,\diff y\right) \left(2 v f(x,v) - \theta F(x, v)\right) \,\diff x \\
		& = I(v) - \frac{1}{\theta} I'(v) v = I(v).
	\end{align*}
	This shows the assertion of the theorem.
\end{proof}

Now we are going to show that problem \eqref{problem} has a ground state solution if we suppose in addition hypothesis \textnormal{(h$_4$)}. For this purpose, we introduce the Nehari manifold associated to problem \eqref{problem} given by
\begin{align*}
	\mathcal{N}= \left\{u\in \V \setminus \{0\}:\; \langle  I'(u), u \rangle_{\mathcal{H}}=0\right \}.
\end{align*}
This means that for any $u\in \mathcal N$,
we have
\begin{align}\label{Nehari-prop}
	& \left(\|\nabla u\|_{p}^p + \|u\|_{p}^p \right) + \left(\|\nabla u\|_{q,a}^q + \|u\|_{q, a}^q \right) =  \int_{\R^N}\left( \int_{\R^N} \frac{F(y, u)}{|x-y|^\mu}\,\diff y\right)f(x,u) u(y)\,\diff x.
\end{align}
Let us define
\begin{align*}
	m:= \inf_{u\in \mathcal N}I(u)
\end{align*}

Now we are ready to prove the existence of a ground state solution under hypotheses \textnormal{(h$_1$)}--\textnormal{(h$_4$)} which completes the proof of Theorem \ref{main_result}.

\begin{proof}[Proof of Theorem \ref{main_result}]
	We are going to show that $I(v)=m$. To this end, for $u \in \mathcal{N}$, we consider the fibering function $\Phi_u\colon (0, \infty) \to \R$ defined by $\Phi_u(t) = I(tu)$ such that
	\begin{align}\label{fiber-map}
		\begin{split}
			\Phi'_u(t) = \langle I'(tu), u \rangle_{\mathcal{H}} 
			&= t^{p-1} \|u\|_{1,p}^p + t^{q-1} \left(\|\nabla u\|_{q,a}^q + \|u\|_{q, a}^q \right) \\
			& \qquad - \int_{\R^N}\left( \int_{\R^N} \frac{F(y,tu)}{|x-y|^\mu}\,\diff y\right)f(x, tu) u\,\diff x.
		\end{split}
	\end{align}
	From \eqref{Nehari-prop} and \eqref{fiber-map}, it follows for $t > 1$ that
	\begin{align}\label{GS:est-1}
		\begin{split}
			\Phi'_u(t) &\leq t^{q-1} \left(\|\nabla u\|_{p}^p + \|u\|_{p}^p + \|\nabla u\|_{q,a}^q + \|u\|_{q, a}^q \right) - \int_{\R^N}\left( \int_{\R^N} \frac{F(y,tu)}{|x-y|^\mu}\,\diff y\right)f(x, tu) u\,\diff x\\
			& = t^{q-1}  \int_{\R^N}\left( \int_{\R^N} \frac{F(y, u)}{|x-y|^\mu}\,\diff y\right)f(x,u) u\,\diff x -  \int_{\R^N}\left( \int_{\R^N} \frac{F(y,tu)}{|x-y|^\mu}\,\diff y\right)f(x, tu) u\,\diff x.
		\end{split}
	\end{align}
	Using \textnormal{(h$_4$)} and Remark \ref{rema:2} in \eqref{GS:est-1}, we obtain
	\begin{align*}
		\frac{\Phi'_u(t)}{ t^{q-1}} &\leq \int_{\R^N}\left( \int_{\R^N} \frac{F(y, u)}{|x-y|^\mu}\,\diff y\right) f(x,u) u\,\diff x -  \int_{\R^N}\left( \int_{\R^N} \frac{F(y,tu)}{|x-y|^\mu}\,\diff y\right)\frac{f(x,tu) |u|^\frac{q}{2}}{t^\frac{q}{2} | tu|^{\frac{q}{2}-2} tu}\,\diff x\\
		& \leq  \int_{\R^N}\left( \int_{\R^N} \frac{F(y, u)}{|x-y|^\mu}\,\diff y\right) f(x,u) u \,\diff x -  \int_{\R^N}\left( \int_{\R^N} \frac{F(y,tu) t^\frac{-q}{2}}{|x-y|^\mu}\,\diff y\right) f(x,u) u\,\diff x \leq 0.
	\end{align*}
	Hence, $\Phi'_u(t) \leq 0$ for $t>1$. Arguing similarly as above for $t<1$, we obtain $\Phi'_u(t) \geq 0$ for $t<1.$ Therefore, the number $1$ is a point of a maximum for the function $\Phi'_u$, that is,
	\begin{align*}
		I(u) = \Phi_u(1) = \max_{t \in [0, \infty]} \Phi_u(t)= \max_{t \in [0, \infty]} I(tu).
	\end{align*}

	Now, we define the map $\gamma\colon  [0,1] \to \V$ as $\gamma(t) = (t_* u) t$ such that $I(t_\ast u) <0$ and $t_\ast>1.$ The mountain pass geometry of the energy functional $I$ implies that the map $\gamma$ is well defined and $\gamma \in \Gamma.$ Hence,
	\begin{align*}
		b \leq \max\limits_{0 \leq t \leq 1} I( \gamma(t)) \leq I((t_\ast u) t_\ast^{-1}) = I(u),
	\end{align*}
	where the second inequality follows by using the fact that $1$ is a point of a maximum of the map $t \to \Phi_u(t)$ for $u \in \mathcal{N}$. Since $u \in \mathcal{N}$ was arbitrary chosen, we deduce
	\begin{align}\label{ineq-1}
		b \leq m.
	\end{align}

	Let $v$ be the solution of problem \eqref{problem} obtained in Theorem \ref{exis:res} such that
	\begin{align}\label{ineq-2}
		I(v) \leq b \quad \text{and} \quad \langle I'(v), \phi \rangle_{\mathcal{H}} =0 \quad \text{for all }  \phi \in \V,
	\end{align}
	where the mountain pass level $b$ is defined in \eqref{MP-level}. Now, by using the fact that $v \in \mathcal{N}$ and combining \eqref{ineq-1} and \eqref{ineq-2}, we obtain the required claim $I(v)=m$.
\end{proof}

\section{Existence of ground state solution without AR-condition}\label{sec_4}

In this section, we establish the existence of a ground state solution of problem \eqref{problem} under the assumptions \textnormal{(h$_1$)}, \textnormal{(h$_2'$)} and \textnormal{(h$_3$)}. We start with the mountain pass geometry.

\begin{lemma}\label{exis-wo-ar-lem1}
	Let hypotheses \textnormal{(h$_1$)}, \textnormal{(h$_2'$)} and \textnormal{(h$_3$)} be satisfied. 
	\begin{enumerate}
		\item[\textnormal{(i)}] 
			There exist $R$, $\sigma>0$ such that $I(u)\geq \sigma$ for all $u\in \V$ with $\|u\|=R$.
		\item[\textnormal{(ii)}] 
			There exists $e\in \V$ with $\|e\|>\sigma$ such that $I(e)<0$.
	\end{enumerate}
\end{lemma}

\begin{proof}

\textnormal{(i)} The proof works in the same way as the one of Lemma \ref{M-P}(i).

\textnormal{(ii)} 	We choose $u\in \V$ such that $u>0$, $\|u\|=1$ and
	\begin{align*}
		\int_{\R^N}\left(\int_{\R^N}\frac{|u(x)|^{\frac{q}{2}}}{|x-y|^\mu}\,\diff x \right)|u(y)|^{\frac{q}{2}}\,\diff y>0.
	\end{align*} 
	For $t>1$ large enough we have
	\begin{align*}
		I(tu) \leq \frac{t^q\rho(u)}{p}- \frac12 \int_{\R^N}\left( \int_{\R^N} \frac{F(y,tu)}{|x-y|^\mu}\,\diff y\right) F(x,tu) \,\diff x.
	\end{align*}
	Moreover,  by Remark \ref{rema:2}, we know that for any $l>0$ there exists $C_l>0$ such that
	\begin{align*}
		F(x,tu)> l|tu(x)|^{\frac{q}{2}},\;\text{when}\; |tu(x)|>C_l
	\end{align*}
	uniformly in $x \in \R^N$. Using this estimate, we obtain \begin{align*}
		I(tu) \leq \frac{t^q\rho(u)}{p}-\frac{l^2t^q}{2} \int_{\R^N}\left(\int_{\R^N}\frac{|u(y)|^{\frac{q}{2}}}{|x-y|^\mu}\,\diff y \right)|u(x)|^{\frac{q}{2}}\,\diff x
	\end{align*}
	when $|tu|>C_l$. Thus, for suitable $l$, we can find $t_*>0$ sufficiently large such that $|t_*u(x)|> C_l$ uniformly for $x \in \R^N$ with $\|t_*u\|>\sigma$ and $I(t_*u)<0$ for some $\sigma>0$. This proves the assertion of the lemma by fixing $e=t_*u$.
\end{proof}

A direct consequence of Lemma \ref{exis-wo-ar-lem1} is the following result.

\begin{corollary}
	Let hypotheses \textnormal{(h$_1$)}, \textnormal{(h$_2'$)} and \textnormal{(h$_3$)} be satisfied. Then there exist $r>0$ and $w\in \V$ such that $\|w\|>r$ and
	\begin{align*}
		A:= \inf_{\|u\|=r}I(u) >I(0)=0\geq I(w).
	\end{align*}
\end{corollary}

\begin{proof}
	Taking Lemma \ref{exis-wo-ar-lem1} into account, we get
	\begin{align*}
		A:= \inf_{\|u\|=\sigma}I(u)\geq R >I(0)=0 > I(e).
	\end{align*}
	The result follows by fixing $r=\sigma$ and $w=e$.
\end{proof}

\begin{lemma}\label{eps-lem}
	Let hypotheses \textnormal{(h$_1$)}, \textnormal{(h$_2'$)} and \textnormal{(h$_3$)} be satisfied. Then there exist $r_0>0$ and $\varepsilon>0$ such that $0<\|u\|<r_0$ implies
	\begin{align*}
		I(u)\geq \varepsilon\|u\|^q
		\quad \text{and}\quad 
		\langle I'(u),u \rangle_\mathcal{H} \geq \varepsilon\|u\|^q.
	\end{align*}
\end{lemma}

\begin{proof}
Similar to the proof of Lemma \ref{M-P} we get
 \begin{align*}
		I(u) \geq \frac{\rho(u)}{q}- C\left(\rho(u)^{\frac{2r_1}{p}}+ \rho(u)^{\frac{2r_2}{p}}\right).
	\end{align*}
By using Proposition  \ref{proposition_modular_properties}(iii), for $0<\|u\|<1$ we have
	\begin{align*}
		\frac{I(u)}{\|u\|^q}\geq 	\frac{I(u)}{\rho(u)}\geq \frac{1}{q}- C\left(\left(\rho(u)\right)^{\frac{2r_1}{p}-1}+ \left(\rho(u)\right)^{\frac{2r_2}{p}-1}\right)
	\end{align*}
	where $2r_2\geq 2r_1>p$ by \eqref{growth:exponent:cond}. This implies that if we choose $r>0$ small enough it follows
	\begin{align*}
		\frac{I(u)}{\|u\|^q}\geq \varepsilon\quad \text{if } 0<\|u\|<r
	\end{align*}
	for some $\varepsilon>0$. Similarly, one can establish
	\begin{align*}
		\frac{\langle I'(u),u\rangle }{\|u\|^q}\geq \varepsilon
		\quad \text{if } 0<\|u\|<r
	\end{align*}
	for some $\varepsilon>0$. This ends the proof. 
\end{proof}

\begin{proposition}\label{cerami-bdd}
	Let hypotheses \textnormal{(h$_1$)}, \textnormal{(h$_2'$)} and \textnormal{(h$_3$)} be satisfied. Then any \textnormal{(C)}$_c$-sequence of $I$ is bounded in $\V$ for any $c\in \R$.
\end{proposition}

\begin{proof}
	We argue indirectly and suppose $\{u_n\}_{n\in\N}$ is an unbounded \textnormal{(C)}$_c$-sequence of $I$. Then, up to a subsequence, we have
	\begin{align*}
		\|u_n\|\to \infty,\quad I(u_n)\to c
		\quad\text{and}\quad \left(1 +\|u_n\|\right) I'(u_n) \to 0.
	\end{align*}
	Let $v_n=\frac{u_n}{\|u_n\|}$, then $\{v_n\}_{n\in\N}$ is bounded in $\V$. Our claim is
	\begin{align}\label{cerami=bdd-3}
		\lim_{n\to \infty}\sup_{y\in \R^N} \int_{B_2(y)}|v_n|^p\,\diff x =0 
	\end{align}
	because if, up to a subsequence,
	\begin{align*}
		\sup_{y\in \mathbb R^n} \int_{B_2(y)}|v_n|^p\,\diff x \geq \delta>0
	\end{align*}
	for some $\delta>0$, then we can choose a sequence $\{z_n\}_{n\in\N}\subset \R^N$ such that
	\begin{align*}
		\int_{B_2(z_n)}|v_n|^p\,\diff x \geq \frac{\delta}{2}.
	\end{align*}
	Since $\mathbb Z^N\cap B_2(z_n)$ can have maximum $4^N$ number of points, we can select $y_n \in \mathbb Z^N\cap B_2(z_n)$ such that
	\begin{align*}
		\int_{B_2(y_n)}|v_n|^p\,\diff x \geq \frac{\delta}{2\times 4^N}:=\tau >0.
	\end{align*}

	Now we set $\tilde{v}_n(\,\cdot\,)= v_n(\cdot +y_n)$ and see that $\rho(v_n)=\rho(\tilde{v}_n)$ due to \textnormal{(h$_3$)}. Hence, $\{\tilde{v}_n\}_{n\in\N}$ is also bounded in $\V$ which implies, up to a subsequence, that
	\begin{align*}
		\tilde{v}_n \rightharpoonup \tilde{v} \quad\text{in } \V,
		\quad \tilde{v}_n \to\tilde{v} \quad \text{in } L^p_{\loc}(\R^N)
		\quad\text{and}\quad\tilde{v}_n(x) \to  \tilde{v}(x)\quad \text{a.\,e.\,in } \R^N
	\end{align*}
	for some $\tilde v \in \V$. From
	\begin{align*}
		\int_{B_2(0)}|\tilde{v}_n|^p\,\diff x =\int_{B_2(y_n)}|v_n|^p\,\diff x \geq \tau>0,
	\end{align*}
	we know  that $\tilde v\not\equiv 0$. We further set $\tilde{u}_n = \tilde{v}_n \|u_n\|$ and have that $|\tilde{u}_n(x)|\to \infty$ if $\tilde{v}(x)\neq 0$. 
	Using \textnormal{(h$_2'$) (iv)}, we get
	\begin{align}\label{cerami-bdd-2}
		\frac{F(x,\tilde u_n(x))|\tilde v_n(x)|^{q/2}}{|\tilde u_n(x)|^{q/2}} \to \infty\quad\text{for all } x\in \Omega,
	\end{align}
	where $\Omega = \{x\in \R^N\,:\, \tilde{v}(x) \neq 0\}$ has positive measure. Since $\lim_{n\to \infty}I(u_n)=c$, we get, for $\|u_n\|>1$, that
	\begin{align*}
		\int_{\R^N}\left( \int_{\R^N} \frac{F(y,u_n)}{|x-y|^\mu}\,\diff y\right) F(x,u_n) \,\diff x 
		&= -(c+o(1)) + \frac{\|u_n\|^p_{1,p}}{p}+\frac{\|\nabla u_n\|_{q,a}^q+\|u_n\|_{q,a}^q}{q} \\
		&\leq -(c+o(1)) + \frac{\rho(u_n)}{p}\\
		& \leq -(c+o(1)) + \frac{\|u_n\|^q}{p}.
	\end{align*}
	Changing the variables and applying \eqref{cerami-bdd-2} gives
	\begin{align*}
		\frac{-(c+o(1))}{\|u_n\|^q}+\frac1p & \geq \int_{\R^N}\left( \int_{\R^N} \frac{F(y,\tilde u_n)}{|x-y|^\mu}\,\diff y\right) \frac{F(x,\tilde u_n)}{\|u_n\|^q} \,\diff x\\
		&=  \int_{\R^N}\left( \int_{\R^N} \frac{F(y,\tilde u_n) |\tilde v_n|^{\frac{q}{2}}}{|x-y|^\mu|\tilde u_n|^{\frac{q}{2}}}\,\diff y\right) \frac{F(x,\tilde u_n)|\tilde v_n|^{\frac{q}{2}}}{|\tilde u_n|^{\frac{q}{2}}} \,\diff x\\
		& \geq \int_{\Omega}\left( \int_{\Omega} \frac{F(y,\tilde u_n) |\tilde v_n|^{\frac{q}{2}}}{|x-y|^\mu|\tilde u_n|^{\frac{q}{2}}}\,\diff y\right) \frac{F(x,\tilde u_n)|\tilde v_n|^{\frac{q}{2}}}{|\tilde u_n|^{\frac{q}{2}}} \,\diff x
		\to \infty.
	\end{align*}
	But this is impossible and so \eqref{cerami=bdd-3} must hold. Then, from \eqref{cerami=bdd-3} and Lions' lemma \cite[Lemma I.1]{Lions-1984} we have
	\begin{align}\label{cerami-bdd-4}
		v_n \to 0 \quad \text{in } L^s(\R^N) \quad \text{for any } s \in (p,p^*).
	\end{align}
	By the continuity of the map $t\mapsto I(tu_n)$ for $t\in [0,1]$ and each fixed $n\in \mathbb N$, we can find  a sequence $\{t_n\}_{n\in\N}\in [0,1]$ such that
	\begin{align}\label{cerami-bdd-5}
		I(t_nu_n)= \max_{t\in [0,1]} I(tu_n). 
	\end{align}
	Next we are going to show that
	\begin{align}\label{cerami-bdd-7}
		\lim_{n\to \infty} I(t_n u_n)=\infty.
	\end{align}
	Since $\|u_n\|\to \infty$, we can choose $M>1$ and $n\in\N$ sufficiently large such that $\frac{M}{\|u_n\|}\in (0,1)$. For such $n$, using \eqref{cerami-bdd-5}, we get
	\begin{align}\label{cerami-bdd-6}
		\begin{split}
			I(t_nu_n) \geq  I\left(\frac{Mu_n}{\|u_n\|}\right) 
			&= I(Mv_n)\\
			& \geq \frac{\min\left\{M^p,M^q\right\}}{q}\rho(v_n) - \int_{\R^N}\left( \int_{\R^N} \frac{F(y,M v_n)}{|x-y|^\mu}\,\diff y\right) {F(x,Mv_n)}\,\diff x\\
			&=\frac{\min\left\{M^p,M^q\right\}}{q} - \int_{\R^N}\left( \int_{\R^N} \frac{F(y,M v_n)}{|x-y|^\mu}\,\diff y\right) {F(x,Mv_n)}\,\diff x,
		\end{split} 
	\end{align}
	since $\|v_n\|=1$ implies $\rho(v_n)=1$, see Proposition \ref{proposition_modular_properties}(ii). Using the estimate as in \eqref{hls-new16} and applying \eqref{cerami-bdd-4} as well as \eqref{growth:exponent:cond}, one has
	\begin{align*}
		\int_{\R^N}\left( \int_{\R^N} \frac{F(y,Mv_n)}{|x-y|^\mu}\,\diff y\right)F(x,Mv_n)\,\diff x &\leq C_1  \left(\int_{\R^N}\left(|Mv_n|^{\frac{2Nr_1}{2N-\mu}}+ |Mv_n|^{\frac{2Nr_2}{2N-\mu}}\right)\,\diff x\right)^{\frac{2N-\mu}{N}}\\
		& \leq C_1M^{2r_2}\left(\int_{\R^N}\left(|v_n|^{\frac{2Nr_1}{2N-\mu}}+ |v_n|^{\frac{2Nr_2}{2N-\mu}}\right)\,\diff x\right)^{\frac{2N-\mu}{N}}\\
		&\to 0 \quad \text{as } n \to \infty.
	\end{align*}
	Inserting this in \eqref{cerami-bdd-6} yields
	\begin{align*}
		I(t_n u_n)\geq \frac{\min\left\{M^p,M^q\right\}}{q}+o_n(1)
	\end{align*}
	which holds for any $M>1$. This proves  \eqref{cerami-bdd-7}. 

	Now we have $I(0)=0$, $\lim_{n\to \infty}I(u_n)=c$ and $\langle I'(t_nu_n),t_nu_n \rangle_{\mathcal{H}} =0$ for $t_n\in (0,1)$. Then, from \eqref{Non-AR-cond} and Remark \ref{rema:2}, it follows that
	\begin{align*}
		\frac{1}{\Theta} I(t_nu_n) &= \frac{1}{\Theta}I(t_nu_n)-\frac{1}{\Theta}\frac{1}{q}\langle I'(t_nu_n),t_nu_n \rangle_{\mathcal{H}} \\
		&= \frac{1}{\Theta} \left[\frac{t_n^p\|u_n\|_{1,p}^p}{p}+ \frac{t_n^q}{q}\left(\|\nabla u_n\|_{q,a}^q+\|u_n\|_{q,a}^q\right) -\frac{1}{q}\rho(t_nu_n) \right.\\
		&\qquad\quad\left. + \frac{1}{2q} \int_{\R^N}\left( \int_{\R^N} \frac{F(y,t_nu_n)}{|x-y|^\mu}\,\diff y\right)(2f(x,t_nu_n)t_nu_n - qF(x,t_nu_n))\,\diff x\right]\\
		&= \frac{1}{\Theta}\left[\frac{t_n^p\|u_n\|_{1,p}^p}{p}+ \frac{t_n^q}{q}\left(\|\nabla u_n\|_{q,a}^q+\|u_n\|_{q,a}^q\right) -\frac{1}{q}\rho(t_nu_n) \right]\\
		&\quad + \frac{1}{2q\Theta} \int_{\R^N}\left( \int_{\R^N} \frac{F(y,t_nu_n)}{|x-y|^\mu}\,\diff y\right)\mathcal F(x,t_nu_n)\,\diff x\\
		&\leq \frac{1}{\Theta}\left[\frac{\|u_n\|_{1,p}^p}{p} -\frac{\|u_n\|_{1,p}^p}{q} \right] + \frac{1}{2q} \int_{\R^N}\left( \int_{\R^N} \frac{F(y,t_nu_n)}{|x-y|^\mu}\,\diff y\right)\mathcal F(x,u_n)\,\diff x\\
		& \leq \frac{\|u_n\|_{1,p}^p}{p} -\frac{\|u_n\|_{1,p}^p}{q} + \frac{1}{2q} \int_{\R^N}\left( \int_{\R^N} \frac{F(y,u_n)}{|x-y|^\mu}\,\diff y\right)
		\mathcal F(x,u_n)\,\diff x\\
		& = \frac{\|u_n\|_{1,p}^p}{p} -\frac{\|u_n\|_{1,p}^p}{q} \\
		&\quad +\frac{1}{2q} \int_{\R^N}\left( \int_{\R^N} \frac{F(y,u_n)}{|x-y|^\mu}\,\diff y\right)(2f(x,u_n)u_n - qF(x,u_n))\,\diff x\\
		&= I(u_n)- \frac{1}{q}\langle I'(u_n),u_n \rangle_{\mathcal{H}} \to c \quad \text{as}\;n\to \infty,
	\end{align*}
	which contradicts \eqref{cerami-bdd-7}. Therefore, we can conclude that $\{u_n\}_{n\in\N}$ must be bounded in $\V$.
\end{proof}

Now we can give the proof of Theorem \ref{main_result-new} which says that problem \eqref{problem} has a nontrivial ground state solution $u_0\in\V$ under hypotheses \textnormal{(h$_1$)}, \textnormal{(h$_2'$)} and \textnormal{(h$_3$)}, that is,
\begin{align*}
	I(u_0)= \inf\left\{I(u)\,:\, u\not \equiv 0 \text{ and } \langle I'(u),u\rangle =0\right\}.
\end{align*}

\begin{proof}[Proof of Theorem \ref{main_result-new}]
	From Lemma \ref{lemma-positive-b} we know that the mountain pass level $c:=b>0$, so there exists a \textnormal{(C)$_c$}-sequence $\{u_n\}_{n\in\N}$ of $I$ at the level $c$.  Moreover, Proposition \ref{cerami-bdd} implies the boundedness of $\{u_n\}_{n\in\N}$  in $\V$. Let us define
	\begin{align*}
		\delta := \lim_{n\to \infty}\sup_{y\in \R^N}\int_{B_2(y)}|u_n|^p\,\diff x.
	\end{align*}
	Suppose $\delta=0$. Then, by Lions' lemma \cite[Lemma I.1]{Lions-1984}, we have $u_n\to 0$ in $L^s(\R^N)$ for any $s\in (p,p^*)$.  Using the estimates in \eqref{hls-new16}, we easily conclude that
	\begin{align}\label{mt2-1}
		\begin{split}
			&\lim_{n\to \infty}\int_{\R^N}\left( \int_{\R^N} \frac{F(y,u_n)}{|x-y|^\mu}\,\diff y\right) F(x,u_n)\,\diff x =0,\\ &\lim_{n\to \infty}\int_{\R^N}\left( \int_{\R^N} \frac{F(y,u_n)}{|x-y|^\mu}\,\diff y\right) f(x,u_n)u_n\,\diff x =0.
		\end{split}
	\end{align}
	Due to $\langle I'(u_n),u_n\rangle_{\mathcal{H}}=0$ and the second limit in \eqref{mt2-1} we get
	\begin{align}\label{mt2-2}
		\|\nabla u_n\|^q_{q,a}+\|u_n\|_{q,a}^q \to 0
		\quad \text{as } n\to \infty.
	\end{align}
	Therefore, using the fact $\{u_n\}_{n\in\N}$ is a \textnormal{(C)$_c$}-sequence along with \eqref{mt2-2} we obtain
	\begin{align*}
		c &= \lim_{n\to \infty} \left(I(u_n) - \frac{1}{p}\langle I'(u_n),u_n\rangle \right)\\
		& = \lim_{n\to \infty} \left( \frac{p-q}{pq}(\|\nabla u_n\|^q_{q,a}+\|u_n\|_{q,a}^q)\right.\\
		&\quad \left.-\frac{1}{2p}  \int_{\R^N}\left( \int_{\R^N} \frac{F(y,u_n)}{|x-y|^\mu}\,\diff y\right)(2f(x,u_n)u_n - pF(x,u_n))\,\diff x \right)\\
		& \to 0 \quad \text{as } n\to \infty,
	\end{align*}
	which is a contradiction to $c>0$. Thus $\delta>0$ and there must be a sequence $\{y_n\}_{n\in\N}\subset \mathbb Z^N$ and a real number $\kappa>0$ such that
	\begin{align}\label{mt2-3}
		\int_{B_2(0)}|v_n|^p\,\diff x = \int_{B_2(y_n)}|u_n|^p\,\diff x >\kappa>0,
	\end{align}
	where $v_n(\,\cdot\,)= u_n(\cdot + y_n)$. Since $\varrho_{\mathcal{H}}(u_n)=\varrho_{\mathcal{H}}(v_n)$ and $(h_3)$ holds, we get $\{v_n\}_{n\in\N}$ to be a bounded sequence in $\V$. Hence there exists $\tilde v \in \V$ such that
	\begin{align*}
		v_n \rightharpoonup \tilde{v}  \quad\text{in } \V
		\quad\text{and}\quad 
		v_n \to \tilde{v}  \quad\text{in } L^p_{\loc}(\R^N).
	\end{align*}
	From \eqref{mt2-3} we know that $\tilde{v}  \neq 0$ a.\,e.\,in $\R^N$. By \textnormal{(h$_3$)}, we assert that $\{v_n\}_{n\in\N}$ is a \textnormal{(C)}$_c$-sequence of $I$ and by Proposition \ref{grad-con}, we have for any $\phi \in C_c^\infty(\R^N)$
	\begin{align*}
		\langle I'(\tilde v),\phi \rangle = \lim_{n\to \infty} \langle I'( v_n),\phi \rangle =0,
	\end{align*}
	which says that $\tilde v$ is a nontrivial solution of \eqref{problem}. Now let
	\begin{align*}
		\alpha = \inf\left\{I(u)\,:\,\; u\not \equiv 0 \text{ and } \langle I'(u),u\rangle=0\right\}.
	\end{align*}
	Let $u$ be an arbitrary critical point of $I$. Then Remark \ref{rema:1} helps to get that
	\begin{align*}
		I(u) & = I(u) - \frac{1}{q} \langle I'(u),u \rangle \\
		& = \frac{q-p}{qp}\|u\|_{1,p}^p + \frac{1}{2q}  \int_{\R^N}\left( \int_{\R^N} \frac{F(y,u)}{|x-y|^\mu}\,\diff y\right)\mathcal F(x,u)\,\diff x\geq 0,
	\end{align*}
	since $q>p$. Thus $\alpha \geq 0$ and $\alpha \leq I(\tilde v)<\infty$. We know that there exists sequence $\{w_n\}_{n\in\N}$ of nontrivial critical points of $I$ such that
	\begin{align*}
		I(w_n)\to \alpha\quad \text{as } n\to \infty.
	\end{align*}
	Since $I'(w_n)=0$, by Lemma \ref{eps-lem}, we can find $r_0>0$ such that
	\begin{align}\label{mt2-4}
		\|w_n\|\geq r_0 \quad \text{for all }n\in\N.
	\end{align}
	Furthermore, we see that
	\begin{align*}
		(1+\|w_n\|) I'(w_n) \to 0\quad \text{as } n\to \infty.
	\end{align*}
	This implies that $\{w_n\}_{n\in\N}$ is a \textnormal{(C)}$_\alpha$-sequence of $I$ at $\alpha$ and Proposition \ref{cerami-bdd} says that $\{w_n\}_{n\in\N}$ must be bounded. Let
	\begin{align*}
		\delta_1 = \lim_{n\to \infty}\sup_{y\in \R^N}\int_{B_2(y)}|w_n|^p\,\diff x.
	\end{align*}
	Then $\delta_1=0$ implies 
	\begin{align*}
		\lim_{n\to \infty}\int_{\R^N}\left( \int_{\R^N} \frac{F(y,w_n)}{|x-y|^\mu}\,\diff y\right) f(x,w_n)w_n\,\diff x =0
	\end{align*}
	by following the same arguments as for \eqref{mt2-1}. This leads to
	\begin{align*}
		\rho(w_n)&= \langle I'(w_n),w_n\rangle + \int_{\R^N}\left( \int_{\R^N} \frac{F(y,w_n)}{|x-y|^\mu}\,\diff y\right) f(x,w_n)w_n\,\diff x \to 0
	\end{align*}
	since $I'(w_n)=0$. Hence, $\|w_n\|\to 0$ as $n\to \infty$ by Proposition \ref{proposition_modular_properties}, which contradicts \eqref{mt2-4}. Therefore $\delta_1> 0$. We set $\tilde{w}_n(\,\cdot\,) = w_n(\cdot+y_n)$. Similar arguments as used before in Propositions \ref{choq-conv} and \ref{grad-con}, for any $\phi \in C_c^\infty(\R^N)$  imply that $I'(\tilde{w}_n)=0$, $I'(w_n)=I'(\tilde{w}_n)\to \alpha$ and for some $\tilde w\in \V$, we have
	\begin{align*}
		\tilde{w}_n \rightharpoonup \tilde{w}\not\equiv 0 \quad \text{in }\V 
		\quad \text{and} \quad 
		\nabla \tilde{w}_n \to \nabla \tilde{w} \quad \text{pointwise a.\,e.\,in }  \R^N.
	\end{align*}
	Obviously $\langle I'(\tilde{w}_n),\phi\rangle =0$ for all $\phi \in C_c^\infty(\R^N)$, that is $\tilde{w}$ is a critical point of $I$. Using again the density of $C_c^\infty(\R^N)$ in $\V$ given in \cite[Proposition 6.4.4]{Harjulehto-Hasto-2019}, we obtain $I'(\tilde{w})=0$. 
	
	It remains to show that $\tilde{w}$ is a ground state solution of problem \eqref{problem}, that is $I(\tilde{w})=\alpha$. Applying Fatou's lemma, the pointwise convergence of $\tilde{w}_n$ and $\nabla \tilde{w}_n$ along with Remarks \ref{rema:1} and \ref{rema:2}, it follows that
	\begin{align*}
		I(\tilde w) & = I(\tilde w) -\frac{1}{q}\langle I'(\tilde w),\tilde w \rangle\\
		& = \frac{q-p}{pq}\|\tilde w\|_{1,p}^p + \frac{1}{2q} \int_{\R^N}\left( \int_{\R^N} \frac{F(y,\tilde w)}{|x-y|^\mu}\,\diff y\right) \mathcal F(x,\tilde w)\,\diff x\\
		&\leq \liminf_{n\to \infty} \left(\frac{q-p}{pq}\|\tilde w_n\|_{1,p}^p + \frac{1}{2q} \int_{\R^N}\left( \int_{\R^N} \frac{F(y,\tilde w_n)}{|x-y|^\mu}\,\diff y\right) \mathcal F(x,\tilde w_n)\,\diff x\right)\\
		& = \liminf_{n\to \infty} \left( I(\tilde w_n) -\frac{1}{q}\langle I'(\tilde w_n),\tilde w_n \rangle\right) = \alpha .
	\end{align*}
	Since $I'(\tilde w)=0$, we obtain $I(\tilde{w})=\alpha$ which finishes the proof.
\end{proof}

 \section*{Acknowledgments}
  A.\,Fiscella is member of the {Gruppo Nazionale per l'Analisi Ma\-tema\-tica, la Probabilit\`a e
 le loro Applicazioni} (GNAMPA) of the {Istituto Nazionale di Alta Matematica ``G. Severi"} (INdAM).
 A.\,Fiscella realized the manuscript within the auspices of the INdAM-GNAMPA project titled "Equazioni alle derivate parziali: problemi e modelli" (Prot\_20191219-143223-545) and of the FAPESP Thematic Project titled "Systems and partial differential equations" (2019/02512-5).


\begin{thebibliography}{1}

\bibitem{Alves-Gao-Squassina-Yang-2017}
	C.O. Alves, F. Gao, M. Squassina, M. Yang,
	{\it Singularly perturbed critical Choquard equations},
	J. Differential Equations {\bf 263} (2017), no. 7, 3943--3988.

\bibitem{Alves-Germano-2018}
	C.O. Alves, G.F. Germano,
	{\it Ground state solution for a class of indefinite variational problems with critical growth},
	J. Differential Equations {\bf 265} (2018), no. 1, 444--477.

\bibitem{Alves-Yang-2014}
	C.O. Alves, M. Yang,
	{\it Existence of semiclassical ground state solutions for a generalized Choquard equation},
	J. Differential Equations {\bf 257} (2014), no. 11, 4133--4164.

\bibitem{Alves-Tavares-2019}
	C.O. Alves, L.S. Tavares,
	{\it A {H}ardy-{L}ittlewood-{S}obolev-type inequality for variable exponents and applications to quasilinear {C}hoquard equations involving variable exponent},
	Mediterr. J. Math. {\bf 16} (2019), Paper No. 55, 27 pp.

\bibitem{Arora-Giacomoni-Mukherjee-Sreenadh-2019}
	R. Arora, J. Giacomoni, T. Mukherjee, K. Sreenadh,
	{\it n-{K}irchhoff-{C}hoquard equations with exponential nonlinearity},
	Nonlinear Anal. {\bf 186} (2019), 113--144.

\bibitem{Arora-Giacomoni-Mukherjee-Sreenadh-2020}
	R. Arora, J. Giacomoni, T. Mukherjee, K. Sreenadh,
	{\it Polyharmonic {K}irchhoff problems involving exponential non-linearity of {C}hoquard type with singular weights},
	Nonlinear Anal. {\bf 196} (2020), 111779, 24 pp.

\bibitem{Baroni-Colombo-Mingione-2015}
	P. Baroni, M. Colombo, G. Mingione,
	{\it Harnack inequalities for double phase functionals},
	Nonlinear Anal. {\bf 121} (2015), 206--222.

\bibitem{Baroni-Colombo-Mingione-2016}
	P. Baroni, M. Colombo, G. Mingione,
	{\it Non-autonomous functionals, borderline cases and related function classes},
	St. Petersburg Math. J. {\bf 27} (2016), 347--379.

\bibitem{Baroni-Colombo-Mingione-2018}
	P. Baroni, M. Colombo, G. Mingione,
	{\it Regularity for general functionals with double phase},
	Calc. Var. Partial Differential Equations {\bf 57} (2018), no. 2, Art. 62, 48 pp.

\bibitem{Biswas-Tiwari-2021}
	R. Biswas, S. Tiwari,
	{\it On a class of {K}irchhoff-{C}hoquard equations involving variable-order fractional {$p(\cdot)$}-{L}aplacian and without {A}mbrosetti-{R}abinowitz type condition},
	Topol. Methods Nonlinear Anal. {\bf 58} (2021), no. 2, 403--439.

\bibitem{Chabrowski-1997} 
	J. Chabrowski, 
	``Variational methods for potential operator equations'',
	Walter de Gruyter \& Co., Berlin, 1997.

\bibitem{Chen-Fiscella-Pucci-Tang-2020}
	S. Chen, A. Fiscella, P. Pucci, X. Tang,
	{\it Semiclassical ground state solutions for critical {S}chr\"{o}dinger-{P}oisson systems with lower perturbations},
	J. Differential Equations {\bf 268} (2020), no. 6, 2672--2716.

\bibitem{Crespo-Blanco-Gasinski-Harjulehto-Winkert-2022}
	\'{A}. Crespo-Blanco, L. Gasi\'{n}ski, P. Harjulehto, P. Winkert,
	{\it A new class of double phase variable exponent problems: existence and uniqueness},
	J. Differential Equations {\bf 323} (2022),  182--228.

\bibitem{DeFilippis-Mingione-ARMA-2021}
	C. De Filippis, G. Mingione, 
	{\it Lipschitz bounds and nonautonomous integrals},
	Arch. Ration. Mech. Anal. {\bf 242} (2021), 973--1057.

\bibitem{Cingolani-Clapp-Secchi-2012}
	S. Cingolani, M. Clapp, S. Secchi,
	{\it Multiple solutions to a magnetic nonlinear {C}hoquard equation},
	Z. Angew. Math. Phys. {\bf 63} (2012), no. 2, 233--248.

 \bibitem{Colombo-Mingione-2015a}
 	M. Colombo, G. Mingione,
 	{\it Bounded minimisers of double phase variational integrals},
 	Arch. Ration. Mech. Anal. {\bf 218} (2015), no. 1, 219--273.

 \bibitem{Colombo-Mingione-2015b}
 	M. Colombo, G. Mingione,
 	{\it Regularity for double phase variational problems},
 	Arch. Ration. Mech. Anal. {\bf 215} (2015), no. 2, 443--496.

\bibitem{Ge-Pucci-2022}
	B. Ge, P. Pucci,
	{\it Quasilinear double phase problems in the whole space via perturbation methods},
	Adv. Differential Equations {\bf 27} (2022), no. 1-2, 1--30.

\bibitem{Ghimenti-VanSchaftingen-2016}
	M. Ghimenti, J. Van Schaftingen,
	{\it Nodal solutions for the {C}hoquard equation},
	J. Funct. Anal. {\bf 271} (2016), no. 1, 107--135.

\bibitem{Harjulehto-Hasto-2019}
	P. Harjulehto, P. H\"{a}st\"{o},
	``Orlicz Spaces and Generalized {O}rlicz Spaces'',
	Springer, Cham, 2019.
	
\bibitem{Jeanjean-1999} 
	L. Jeanjean, 
	{\it On the existence of bounded Palais-Smale sequences and
	application to a Landesman-Lazer-type problem set on $\mathbb{R}^N$},
	Proc. Roy. Soc. Edinburgh Sect. A {\bf 129} (1999), no. 4, 787--809.

\bibitem{Hou-Ge-Zhang-Wang-2020}
	G. Hou, B. Ge, B. Zhang, L. Wang,
	{\it Ground state sign-changing solutions for a class of double-phase problem in bounded domains},
	Bound. Value Probl. {\bf 2020}, Paper No. 24, 21 pp.

\bibitem{Le-2022}
	P. Le,
	{\it Liouville results for double phase problems in $\mathbb{R}^N$},
	Qual. Theory Dyn. Syst. {\bf 21} (2022), no. 3, Paper No. 59, 18 pp.

\bibitem{Lieb-1976-77}
	E.H. Lieb,
	{\it Existence and uniqueness of the minimizing solution of {C}hoquard's nonlinear equation},
	Studies in Appl. Math. {\bf 57} (1976/77) no. 2, 93--105.

\bibitem{Lieb-Loss-2001}
	E.H. Lieb, M. Loss,
	``Analysis'', second edition,
	American Mathematical Society, Providence, RI, 2001.

\bibitem{Lions-1980}
	P.-L. Lions,
	{\it The {C}hoquard equation and related questions},
	Nonlinear Anal. {\bf 4} (1980), no. 6, 1063--1072.

\bibitem{Lions-1984} 
	P.-L. Lions,
	{\it The concentration-compactness principle in the calculus of
	variations. The locally compact case. I},
	Ann. Inst. H. Poincar\'{e} Anal. Non Lin\'{e}aire {\bf 1} (1984), no. 4, 223--283.

\bibitem{Liu-2010}
	S. Liu,
	{\it On ground states of superlinear $p$-Laplacian equations in $\mathbb{R}^N$},
	J. Math. Anal. Appl. {\bf 361} (2010), no. 1,48--58.

\bibitem{Liu-Dai-2020}
	W. Liu, G. Dai,
	{\it Multiplicity results for double phase problems in {$\mathbb{R}^N$}},
	J. Math. Phys. {\bf 61} (2020), no. 9,  Art. 091508, 20 pp.

\bibitem{Liu-Dai-2018}
	W. Liu, G. Dai,
	{\it Three ground state solutions for double phase problem},
	J. Math. Phys. {\bf 59} (2018), no. 12, 121503, 7 pp.

\bibitem{Liu-Winkert-2022}
	W. Liu, P. Winkert,
	{\it Combined effects of singular and superlinear nonlinearities in singular double phase problems in $\mathbb{R}^N$},
	J. Math. Anal. Appl. {\bf 507} (2022), no. 2, Paper No. 125762, 19 pp.

\bibitem{Ma-Zhao-2010}
	L. Ma, L. Zhao,
	{\it Classification of positive solitary solutions of the nonlinear Choquard equation},
	Arch. Ration. Mech. Anal. {\bf 195} (2010), no. 2, 455--467.

\bibitem{Marcellini-1991}
	P. Marcellini,
	{\it Regularity and existence of solutions of elliptic equations with {$p,q$}-growth conditions},
	J. Differential Equations {\bf 90} (1991), no. 1, 1--30.
	
\bibitem{Marcellini-1989b}
	P. Marcellini,
	{\it Regularity of minimizers of integrals of the calculus of variations with nonstandard growth conditions},
	Arch. Rational Mech. Anal. {\bf 105} (1989), no. 3,  267--284.

\bibitem{Mingqi-Radulescu-Zhang-2019}
	X. Mingqi, V.D. R\u{a}dulescu, B. Zhang,
	{\it A critical fractional {C}hoquard-{K}irchhoff problem with magnetic field},
	Commun. Contemp. Math. {\bf 21} (2019), no. 4, 1850004, 36 pp.

\bibitem{Moroz-Penrose-Tod-1997}
	I.M. Moroz, R. Penrose, P. Tod,
	{\it Spherically-symmetric solutions of the {S}chr\"{o}dinger-{N}ewton equations},
	Classical Quantum Gravity {\bf 15} (1998), no. 9, 2733--2742.

\bibitem{Moroz-VanSchaftingen-2017}
	V. Moroz, J. Van Schaftingen,
	{\it A guide to the Choquard equation},
	J. Fixed Point Theory Appl. {\bf 19} (2017), no. 1, 773--813.

\bibitem{Moroz-VanSchaftingen-2015}
	V. Moroz, J. Van Schaftingen,
	{\it Existence of groundstates for a class of nonlinear Choquard equations},
	Trans. Amer. Math. Soc. {\bf 367} (2015), no. 9, 6557--6579.

\bibitem{Moroz-VanSchaftingen-2013}
	V. Moroz, J. Van Schaftingen,
	{\it Groundstates of nonlinear Choquard equations: existence, qualitative properties and decay asymptotics},
	J. Funct. Anal. {\bf 265} (2013), no. 2, 153--184.

\bibitem{Moroz-VanSchaftingen-2015b}
	V. Moroz, J. Van Schaftingen,
	{\it Groundstates of nonlinear {C}hoquard equations: Hardy-Littlewood-Sobolev critical exponent},
	Commun. Contemp. Math. {\bf 17} (2015), no. 5, 1550005, 12 pp.

\bibitem{Mukherjee-Sreenadh-2017}
	T. Mukherjee, K. Sreenadh,
	{\it Fractional {C}hoquard equation with critical nonlinearities},
	NoDEA Nonlinear Differential Equations Appl. {\bf 24} (2017), no. 6, Paper No. 63, 34 pp.

\bibitem{Pekar-1954}
	S. Pekar, 
	{\it Untersuchung u\"uber die Elektronentheorie der Kristalle}, Akademie Verlag, Berlin, 1954.

\bibitem{Simon-1978}
	J. Simon,
	{\it R\'{e}gularit\'{e} de la solution d'une \'{e}quation non lin\'{e}aire dans {${\mathbb{R}}^{N}$}},
	Journ\'{e}es d'{A}nalyse {N}on {L}in\'{e}aire ({P}roc. {C}onf. {B}esan\c{c}on, 1977), Springer, Berlin {\bf 665} (1978), 205--227.

\bibitem{Steglinski-2022}
	R. Stegli\'{n}ski,
	{\it Infinitely many solutions for double phase problem with unbounded potential in $\mathbb{R}^N$},
	Nonlinear Anal. {\bf 214} (2022), Paper No. 112580, 20 pp.

\bibitem{Sun-Chang-2022}
	W. Sun, X. Chang,
	{\it Existence of least energy nodal solutions for a double-phase problem with nonlocal nonlinearity},
	Appl. Anal., https://doi.org/10.1080/00036811.2021.1999422.
	
\bibitem{Zhikov-1986}
	V.V. Zhikov,
	{\it Averaging of functionals of the calculus of variations and elasticity theory},
	Izv. Akad. Nauk SSSR Ser. Mat. {\bf 50} (1986), no. 4, 675--710.
	
\bibitem{Zhikov-1995}
	V.V. Zhikov,
	{\it On Lavrentiev's phenomenon},
	Russian J. Math. Phys. {\bf 3} (1995), no. 2, 249--269.
	
\bibitem{Zhikov-2011}
	V.V. Zhikov,
	{\it On variational problems and nonlinear elliptic equations with nonstandard growth conditions},
	J. Math. Sci. {\bf 173} (2011), no. 5, 463--570.

\bibitem{Zuo-Fiscella-Bahrouni-2022}
	J. Zuo, A. Fiscella, A. Bahrouni,
	{\it Existence and multiplicity results for {$p(\cdot)\& q(\cdot)$} fractional {C}hoquard problems with variable order},
	Complex Var. Elliptic Equ. {\bf 67} (2022), no. 2, 500--516.

\end{thebibliography}
\end{document}